\documentclass[12pt]{article}
\usepackage[cp1251]{inputenc}
\usepackage[russian, english]{babel}
\usepackage{amsfonts,amsmath,amsthm}
\usepackage{caption}
\usepackage{float}
\usepackage{tikz-cd}

\usepackage[left=1.5cm, right=1.5cm, top=2cm, bottom=2cm, headsep=0.7cm, footskip=1cm]{geometry}

\newtheorem{theorem}{Theorem}
\newtheorem{lemma}{Lemma}

\newtheorem{remark}{Remark}

\begin{document}

\title{Integrable Geodesic Flows on Cones over Riemannian Manifolds}

\author {Andrey E. Mironov and Siyao Yin}

\date{}
\maketitle
\begingroup
\renewcommand\thefootnote{}
\footnote{The work is supported by the Mathematical Center in Akademgorodok under the agreement No.~075-15-2025-348 with the Ministry of Science and Higher Education of the Russian Federation.}%
\addtocounter{footnote}{-1}
\endgroup

\begin{abstract}
  In this paper we study the behavior of geodesics on cones over arbitrary $C^3$-smooth closed Riemannian manifolds. 
  We show that the geodesic flow on such cones admits first integrals whose values uniquely determine almost all geodesics  except for radial geodesics; thus, the geodesic flow is superintegrable. 
  Moreover, we prove that the geodesic flow restricted to the open dense subset of the cotangent bundle corresponding to all non-radial trajectories is Liouville--Arnold integrable.
  This investigation is inspired by our recent results on Birkhoff billiards inside cones over convex manifolds where similar results hold true.
\end{abstract}


\section{Introduction}
Let $\Gamma$ be a closed smooth Riemannian manifold, $\operatorname{dim}\Gamma=n$, with the metric $ds^2=\sum_{i,j=1}^ng_{ij}(u)du^idu^j,$ where $u=(u^1,\dots,u^n)$ are local coordinates on $\Gamma$. The geodesics on $\Gamma$ are defined by the equations
$\ddot{u}^i+\Gamma^i_{jk}\dot{u}^j\dot{u}^k=0$,
where $\Gamma^i_{jk}$ are Christoffel symbols. In some special cases these equations can be integrated. For example, Jacobi solved the equations in the case of ellipsoids. In the general case solving these equations is a hard problem.
Denote by $T^1\Gamma\subset T\Gamma$ the unit tangent bundle.
Let $x_0\in\Gamma, v_0\in T^1_{x_0}\Gamma$ and let $\gamma(s)$ be the unique geodesic on $\Gamma$
such that $\gamma(0)=x_0, \dot{\gamma}(0)=v_0$.
The geodesic flow $\rho^s:T^1\Gamma\rightarrow T^1\Gamma$ is defined by the identity
\[
  \rho^s(x_0,v_0)=(\gamma(s),\dot{\gamma}(s)).
\]
A function $F:T^1\Gamma\rightarrow {\mathbb R}$ is called a {\it first integral}
if it is invariant under the flow.
The geodesic flow can also be equivalently defined by the Hamiltonian system on $T^*\Gamma$
\[
  \dot{u}^i=\frac{\partial H}{\partial p_i},\qquad
  \dot{p}_i=-\frac{\partial H}{\partial u^i},
\]
where $H=\frac{1}{2}\sum_{i,j=1}^n g^{ij}(u)p_ip_j$.
A function $G:T^*\Gamma\rightarrow {\mathbb R}$ is called a {\it first integral} if
\[
 \frac{d G}{ds}
 =\{H,G\}
 =\sum_{i=1}^n\left(
 \frac{\partial G}{\partial u^i}\frac{\partial H}{\partial p_i}
 -
 \frac{\partial H}{\partial u^i}\frac{\partial G}{\partial p_i}
 \right)=0.
\]
The geodesic flow is called {\it Liouville--Arnold integrable}
if there are first integrals $G_1=H,G_2,\dots,G_n$ functionally independent almost everywhere
and $\{G_i,G_j\}=0$, $i,j=1,\dots,n$;
in this case, the geodesic equations can be integrated by quadratures.

There are many remarkable examples of Riemannian manifolds with integrable geodesic flows, see, e.g., \cite{BKF}--\cite{BF} and references therein.
There exist topological obstructions for integrability in the case of real analytic metrics.
Kozlov's theorem establishes that the geodesic flow on an oriented closed surface of genus $g>1$ with any real analytic metric does not admit real analytic first integrals \cite{K}.
Taimanov \cite{T} generalized this result to the  many-dimensional case. 
Butler \cite{But} showed that there exist real-analytic Riemannian metrics on certain compact nilmanifolds for which the geodesic flow does not have real-analytic first integrals, while remaining integrable in the smooth category.
Bolsinov and Taimanov constructed an example of real-analytic metric on a $3$-dimensional manifold whose geodesic flows are smoothly integrable and yet have positive topological entropy \cite{BT}.
In general, the problem of the existence of metrics on an arbitrary manifold with integrable geodesic flows is very interesting and difficult. 
For example, it is unclear whether there exist metrics on closed surfaces of genus $g>1$ with Liouville--Arnold integrable geodesic flow in the smooth category.
It is a very interesting question whether there exist metrics on the two-dimensional torus whose geodesic flows admit non-reducible polynomial first integrals in 
$p_1, p_2$
of degree greater than two (see, e.g., \cite{BM}).

In this paper we study the geodesic flow on the cone over an arbitrary closed $C^3$ Riemannian manifold $\Gamma$, 
$\dim \Gamma =n$.  
By Nash embedding theorem, $\Gamma$ can be isometrically embedded in 
$\mathcal{P} = \{ x^{N+1} = 1 \} \subset \mathbb{R}^{N+1}$ for some $N>n$.  
Let us consider the cone over $\Gamma$ 
$$
K = \{\, t p \mid p \in \Gamma, \; t \in \mathbb{R} \,\}
  \subset \mathbb{R}^{N+1},
$$
equipped with the Riemannian metric induced by the ambient Euclidean metric of $\mathbb{R}^{N+1}$.
We denote its upper and lower parts by
$$
K^+ = \{\, t p \mid p \in \Gamma, \; t>0 \,\}, 
\qquad
K^- = \{\, t p \mid p \in \Gamma, \; t<0 \,\}.
$$
Let $O$ be the singular point of $K$ ($t=0$). 
To define $TK$ it is natural to set  
$T_O K := K$.
We show below that if a geodesic on $K$ contains $O$ or has $O$ as a limit point, then the geodesic is a cone generatrix,
i.e. a line passing through $O$ (see Lemma~\ref{lem:two-types}).
Hence the geodesic flow
\[
\rho^s : T^1K \to T^1K, \qquad s \in \mathbb{R},
\]
is defined for all $s$.

\begin{figure}[h]
  \begin{center}
  \includegraphics[scale=0.245]{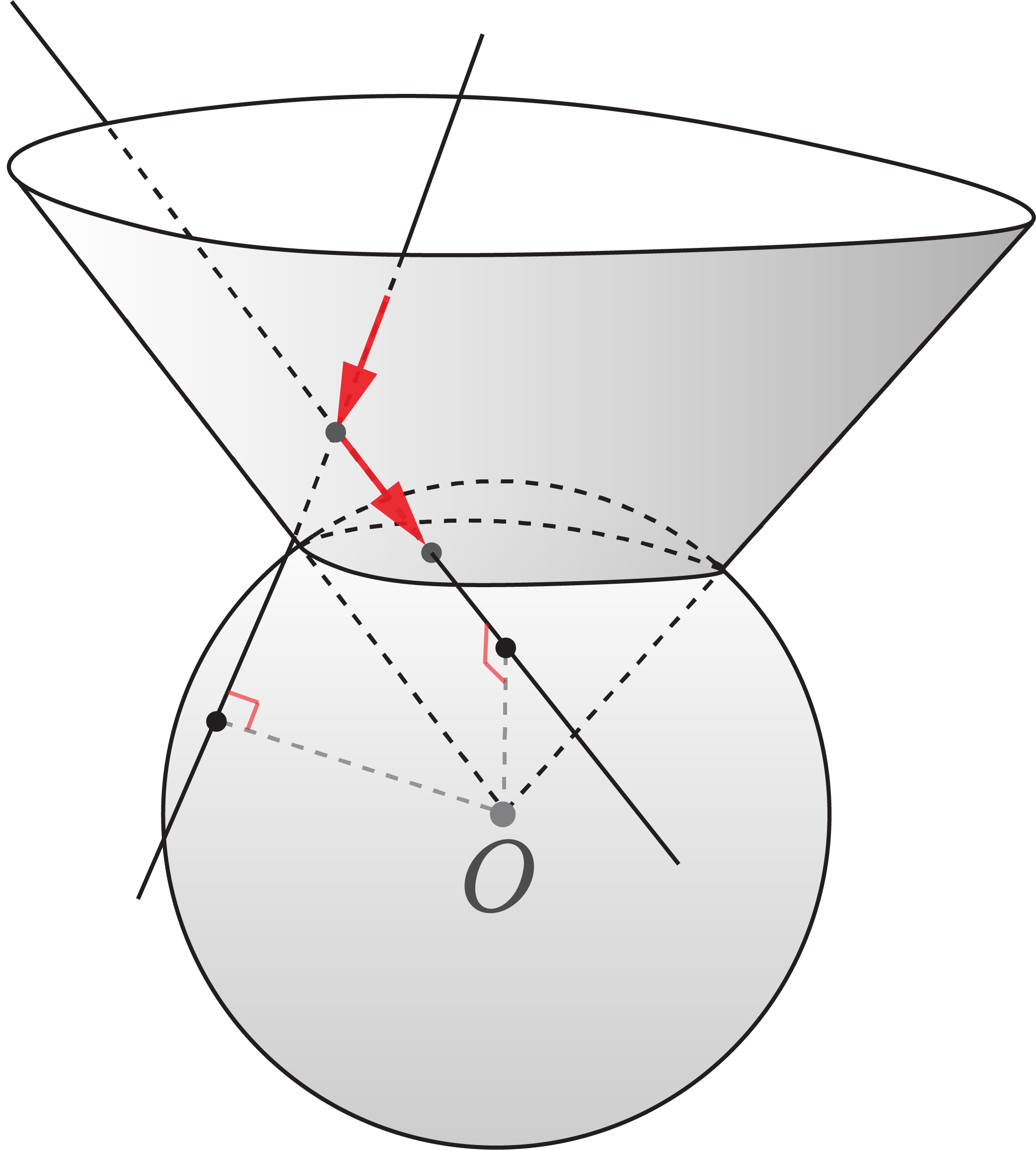}
  \end{center}
  \caption{The sphere as a caustic of the billiard inside a cone.}
  \label{fig:caustic}
\end{figure} 

In \cite{MY1,MY2} the billiard dynamics inside $K^+$ was studied in the case where 
$\dim \Gamma= N-1$.  
It was shown that the line containing any segment of a billiard trajectory in $K^+$  
is tangent to a fixed sphere centered at $O$ (see Fig.~1).  
Hence the radius $r$ of this sphere serves as a first integral of the billiard system.
It turns out that an analogous property holds for geodesics on a cone.  
For any non-generatrix geodesic on $K$, all its tangent lines are tangent to a common sphere (see Fig.~2).

\begin{figure}[H]
  \begin{center}
  \includegraphics[scale=0.245]{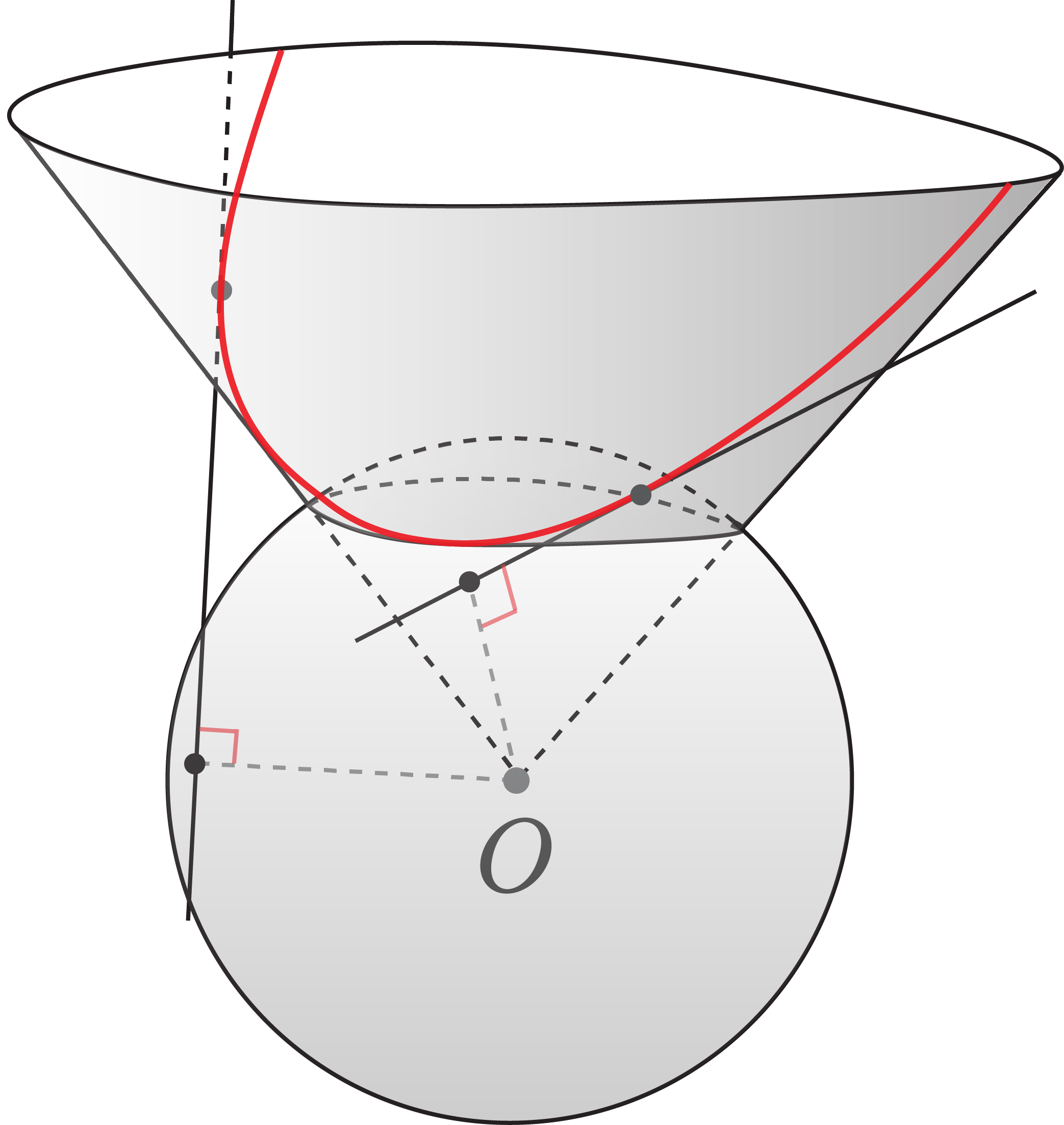}
  \end{center}
  \caption{All tangent lines of a given non-generatrix geodesic are tangent to a common sphere.}
  \label{fig:caustic-geo}
\end{figure} 

We have the following.

\begin{lemma}\label{lem:distance}
  Let $\gamma(s)$ be a geodesic on a cone $K \subset \mathbb{R}^{N+1}$. 
  Then 
  \begin{equation}\label{eq:distance-integral}
  I = \|\gamma(s)\|^2 - \frac{\langle \gamma(s), \gamma'(s) \rangle^2}{\|\gamma'(s)\|^2}
  \end{equation}
  remains constant along $\gamma$.  
  Geometrically, $I$ represents the squared distance from the vertex $O$ to the tangent line of~$\gamma$.
  \end{lemma}
  \begin{remark}
    The first integrals  $I$ in~\eqref{eq:distance-integral} coincides with the restriction to $TK$ of billiard first integral $I_B$ obtained in \cite{MY1} for the billiard inside $K^+$. 
    In that case, the corresponding first integral has the form 
    \[
    I_B = \| x\|^2- \frac{\langle x, v \rangle^2}{\| v\|^2},
    \]
    where $x=(x^1,\dots,x^{N+1})$ denotes the position of the particle and 
    $v=(v^1,\dots,v^{N+1})$ its velocity.
    \end{remark}

Moreover, for billiards, if $\Gamma$ is $C^3$-smooth and strictly convex with an everywhere nondegenerate second fundamental form,  
the system admits a set of first integrals whose values uniquely determine every trajectory.  
We extend this idea to the geodesic flow on $K$, constructing a set of first integrals that uniquely determine all non-radial geodesics.  
The first main result of the paper is the following.
\begin{theorem}\label{lem:integral}
    There exist continuous first integrals $I^1,\ldots,I^{2N+2}:T^1K\to \mathbb{R}$ of the geodesic flow on $T^1K$, which are $C^1$-smooth on $T^1K^{+}\cup T^1K^{-}$, such that the image of the map
    $$
    \mathcal{I}=(I^1,\ldots,I^{2N+2}):T^1K\to\mathbb{R}^{2N+2}
    $$
    has the property that each point in the image, except the origin in $\mathbb{R}^{2N+2}$, determines a unique geodesic on $K$, while the origin corresponds to all radial geodesics on $K$.
\end{theorem}

Although it is unclear when the geodesic flow on an arbitrary Riemannian manifold $\Gamma$ is integrable, on the cone 
over $\Gamma$ the geodesic flow is superintegrable in the sense that it admits enough first integrals to uniquely define almost all geodesics except for cone generatrices.

In the next lemma, we state several properties of geodesics on $K$. Let 
\[
  \Sigma = K \cap \mathbb{S}^N,
  \qquad 
  \dim \Sigma = n,
\] 
where $\mathbb{S}^N\subset \mathbb{R}^{N+1}$ is the unit sphere centered at $O$ and let $g_{\scriptscriptstyle \Sigma}$ denote the induced metric on $\Sigma$.
The metric on the cone $K$ has the form 
$$  
g = dt^2 + t^2 g_{\scriptscriptstyle \Sigma},
$$
which is an example of a warped product metric, studied for example in \cite{ON}.
In particular, Proposition 7.38 together with Remark 7.39 in \cite{ON} imply that the projection of any geodesic in 
$K$ onto $\Sigma$  is a pregeodesic on $\Sigma$ 
(that is, a geodesic up to reparametrization).
In the following lemma we give an explicit formulation of this correspondence in the case of the cone and analyze the asymptotic behavior of the geodesic flow as $s\to \pm\infty$.

Suppose $\gamma(s)\subset K$ is a non-radial geodesic. Then $I>0$ and $s\in(-\infty,+\infty)$ 
(see Lemma~\ref{lem:two-types} below). One observes that $a_1 \gamma(a_2 s)$, 
$a_1, a_2\in \mathbb{R}\setminus\{0\}$, is also a geodesic. Thus, without loss of generality, in the next lemma 
we assume that $I=1$ and $\|\gamma'(s)\|=1$. We have:

\begin{lemma}\label{lem:cone-sphere-correspondence}
  \begin{itemize}
    \item[1)] The geodesic $\gamma(s)$ is tangent to $\mathbb{S}^N$ at a unique point $\gamma(s_0)$.  
    Its radial projection
    \begin{equation}\label{eq:radialprojection}
      \tilde{\gamma}(\tilde{s}) = \frac{\gamma(s)}{\|\gamma(s)\|} \subset \Sigma,
    \end{equation}
    where
    \begin{equation}\label{eq:parameter}
      s = \tan\tilde{s},
      \qquad 
      \tilde{s}\in\left(-\frac{\pi}{2},\frac{\pi}{2}\right),
    \end{equation}
    is a geodesic in $\Sigma$, and $\tilde{s}$ is the arc-length parameter of $\tilde{\gamma}$, $\|\tilde{\gamma}'(\tilde{s})\|=1$.

    \item[2)] Conversely, let $\tilde{\gamma}(\tilde{s})$,  
    $\tilde{s}\in\left(-\frac{\pi}{2},\frac{\pi}{2}\right)$, be an interval of a geodesic on $\Sigma$ of length $\pi$, with $\tilde{s}$ as its arc-length parameter. Then
    \begin{equation}\label{eq:lift}
      \gamma(s) = \tilde{\gamma}(\tilde{s})\sqrt{s^{2}+1}, 
      \qquad
      \tilde{s} = \arctan s,
      \qquad
      s\in(-\infty,+\infty),
    \end{equation}
    is a geodesic on $K$, with $s$ as its arc-length parameter, $\|\gamma'(s)\|=1$, touching  
    $\mathbb{S}^N$ at $\tilde{\gamma}(0)=\gamma(0)$.

    \item[3)] The following limits exist:
    \[
    \lim_{s\to+\infty}{\gamma'(s)}
    =
    \lim_{s\to+\infty}\frac{\gamma(s)}{\|\gamma(s)\|}
    =
    \tilde{\gamma}\left(\frac{\pi}{2}\right)\in K,
    \]
    \[
    \lim_{s\to-\infty}{\gamma'(s)}
    =
    -\lim_{s\to-\infty}\frac{\gamma(s)}{\|\gamma(s)\|}
    =
    -\tilde{\gamma}\left(-\frac{\pi}{2}\right)\in K.
    \]
  \end{itemize}
\end{lemma}
\begin{remark}\label{rmk:covering}
  If $\tilde{\gamma}$ lies on a closed geodesic in $\Sigma$, the interval
  $\tilde{s}\in(-\frac{\pi}{2},\frac{\pi}{2})$ may cover the same closed geodesic more than once.
  Thus, the projection $\gamma(\tan \tilde s) \mapsto \tilde \gamma(\tilde s)$ should not be understood as an one-to-one map in general.
\end{remark}

 The correspondence in Lemma~\ref{lem:cone-sphere-correspondence} shows that non-radial geodesics on $K$
are completely determined, up to a radial scaling, by geodesic segments of fixed length on $\Sigma$,
which is the underlying reason why the geodesic flow on $K$ admits the set of first integrals
stated in Theorem~\ref{lem:integral}.

\medskip

In the next theorem, we will prove that the geodesic flow is Liouville--Arnold integrable when restricted to the open dense subset $\mathcal{M} \subset T^*K^+$ consisting of all cotangent vectors associated with non-radial geodesics.
To describe $\mathcal{M}$,
 let us introduce local coordinates $(t, z)$ on $K^{+}$, where $t$ is the radial coordinate and $z$ represents coordinates on $\Sigma$. 
 Let $(p_0, p)$ be the coordinates on the cotangent fibers induced by the local coordinates $(t, z)$.
 Note that radial geodesics are characterized by $z$ being constant (i.e., the tangent vector is purely radial), which implies that the dual coordinates $p$ vanish. Therefore, the condition for a cotangent vector to correspond to a non-radial geodesic is the non-vanishing of  $p$; 
 hence the set $\mathcal{M}$ is characterized by
$$
\mathcal{M} = \{(t, z; p_0, p) \in T^*K^+ \mid p \neq 0\}.
$$
The restriction of the geodesic flow to $\mathcal{M}$ is smooth and complete (see Lemma~\ref{lem:two-types}). 
Let $H$ be the Hamiltonian generating the geodesic flow. 
The second main result of the paper is the following.
\begin{theorem}\label{thm:liouville-integrability}
  The geodesic flow restricted to the open dense subset $\mathcal{M} \subset T^*K^+$ is Liouville--Arnold integrable; 
  that is, there exist $n+1$ $C^{1}$-smooth first integrals $\mathcal{F}_0=H, \mathcal{F}_1, \dots, \mathcal{F}_{n}$ on $\mathcal{M}$ that are pairwise in involution, i.e., $\{\mathcal{F}_i, \mathcal{F}_j\}=0$, and functionally independent almost everywhere.
\end{theorem}

\begin{remark}\label{eq:cr-smooth}
  If $\Gamma$ is a $C^r$ Riemannian manifold, 
then the first integrals $\mathcal{F}_1, \dots, \mathcal{F}_{n}$ are of class $C^{r-2}$.
\end{remark}

By our construction, the first integrals $\mathcal{F}_1, \dots, \mathcal{F}_{n}$ can not be defined on the subset of $T^*K^+$ that corresponds to the radial geodesics.

In section~2 we will prove Lemma~\ref{lem:distance} and Lemma~\ref{lem:cone-sphere-correspondence}. 
In section~3 we will prove Theorem~\ref{lem:integral}.
In section~4 we will prove Theorem~\ref{thm:liouville-integrability}.

\section{Geodesics on Cones}
We start by proving 
Lemma~\ref{lem:distance} (see Introduction). 
Lemma~\ref{lem:distance} provides an essential observation of the geodesic flow on a cone; its proof is straightforward.
\begin{proof}[Proof of Lemma \ref{lem:distance}]
  Since $\gamma(s)$ is a geodesic on $K$, we have $ \frac{d}{ds}\left(\|\gamma'(s)\|^2\right) =0$. 
  Moreover, since $\gamma''(s)$ is orthogonal to the cone, we have $\langle \gamma''(s), \gamma(s) \rangle = 0$. 
  Therefore,
  $$
      \frac{d}{ds}  \left(
        \|\gamma(s)\|^2- \frac{\langle \gamma(s), \gamma'(s) \rangle^2}{\|\gamma'(s)\|^2}
        \right) 
  = \frac{d \|\gamma(s)\|^2 }{ds}  - 
  \frac{d  \langle \gamma(s), \gamma'(s) \rangle^2}{ds}
  \frac{1}{\|\gamma'(s)\|^2}
  $$
  $$        
        = 2 \langle \gamma(s), \gamma'(s) \rangle
        - 
        2\langle \gamma(s), \gamma'(s) \rangle
         \Big(\langle \gamma'(s), \gamma'(s) \rangle + \langle \gamma(s), \gamma''(s) \rangle   \Big)
         \frac{1}{\|\gamma'(s)\|^2}
         =0.
  $$
  Hence $I$ is a first integral.
  
  Geometrically, $\|\gamma(s)\|^2$ is the squared length of $\gamma(s)$,
  and $\langle \gamma(s), \gamma'(s)\rangle^2 / \|\gamma'(s)\|^2$ is  the squared length of its projection onto the tangent direction.
  Their difference therefore equals the squared distance from $O$ to the tangent line.

  Lemma \ref{lem:distance} is proved.
  \end{proof}

By  Lemma \ref{lem:distance}, \( I \ge 0 \). 
For our convenience before proving Lemma~\ref{lem:cone-sphere-correspondence} we first prove the following lemma, 
which classifies geodesics on the cone according to whether \( I = 0 \) or \( I > 0 \).

\begin{lemma}\label{lem:two-types}
  Every geodesic $\gamma(s)$ on $K$ with $\|\gamma'(s)\|=1$, after an appropriate shift of parameter, belongs to one of the following two classes:
  \begin{enumerate}
    \item[(i)] \textbf{Radial geodesics (generatrices):}  
    $$
      \gamma(s) = s p, 
      \quad p \in \Sigma, 
      \quad  s\in (-\infty,+\infty),
    $$
    which are straight lines passing through the vertex $O$. They satisfy $I = 0$.
    
    \item[(ii)] \textbf{Non-radial geodesics:}  
    \begin{equation}\label{eq:nonradial-geodesic}
      \gamma(s) = q(s)\sqrt{s^2 + I}, \quad s\in (-\infty,+\infty),
    \end{equation}
    where $q(s)$ is some curve on $\Sigma$. They satisfy $I > 0$.
  \end{enumerate}
  \end{lemma}

\begin{proof}

  If $I=0$, then by \eqref{eq:distance-integral} we have 
  $\|\gamma(s)\|^2 = \langle \gamma(s), \gamma'(s)\rangle^2$, 
  which means that $\gamma(s)$ is everywhere parallel to its tangent direction $\gamma'(s)$. 
  Hence $\gamma$ is a generatrix of the cone.
  
  If $I>0$, then by Lemma~\ref{lem:distance} we have
  $\|\gamma(s)\|^2 \geq I > 0$. 
  Therefore the origin $O$ is not a limit point of $\gamma(s)$. 
  In particular, the geodesic lies entirely in one of the two connected components $K^+$ or $K^-$ 
  and stays at a distance not smaller than $\sqrt{I}$ from the vertex $O$.
  
  Assume that $\gamma$ lies in $K^+$. 
  We can modify the cone near the vertex to obtain a complete smooth Riemannian manifold 
  that coincides with $K^+$ outside the sphere of radius $\sqrt{I}/2$ centered at $O$.
  This modified manifold is complete as a metric space. 
  By the Hopf--Rinow theorem, it is also geodesically complete; 
  hence the maximal interval of existence of $\gamma$ is $(-\infty, \infty)$.
  
  To derive the representation \eqref{eq:nonradial-geodesic}, 
  we write $\gamma(s)$ in the form
  \[
  \gamma(s) = q(s)t(s),
  \]
  where $q \in \mathbb{S}^N$ and $t(s) = \|\gamma(s)\| > 0$.
  Then
  \[
  t^2(s) = \langle \gamma(s), \gamma(s)\rangle.
  \]
  Differentiating twice with respect to $s$, we obtain
  \[
  \frac{d^2}{ds^2} t^2(s)
    = 2\langle \gamma'(s), \gamma'(s)\rangle 
      + 2\langle \gamma''(s), \gamma(s)\rangle.
  \]
  Since $\gamma(s)$ is a geodesic on $K$, it satisfies 
  $\langle \gamma''(s), \gamma(s)\rangle = 0$,
  and by our assumption $\|\gamma'(s)\| = 1$.
  Therefore,
  \[
  \frac{d^2}{ds^2} t^2(s) = 2.
  \]
  Integrating twice yields
  \[
  t^2(s) = s^2 + c_1 s + c_2,
  \]
  where $c_1, c_2$ are constants. 
  By shifting the parameter $s$, we may assume $c_1 = 0$, 
  so that
  \[
  t^2(s) = s^2 + c_2,
  \]
  where $c_2$ is a positive constant. 

  The function $t^2(s)$ attains its minimum at $s = 0$, 
where $\gamma(0) = \sqrt{c_2} q(0)$ and $\gamma'(0) = \sqrt{c_2} q'(0)$.
Since $\|q(s)\| = 1$, it follows that $\langle q(0), q'(0)\rangle = 0$. 
Substituting these relations into the definition of $I$ in \eqref{eq:distance-integral}, 
we find that $c_2=I$. 
Hence we obtain the representation \eqref{eq:nonradial-geodesic}.

Lemma \ref{lem:two-types} is proved.
  \end{proof}

\begin{proof}[Proof of Lemma \ref{lem:cone-sphere-correspondence}]
    1)
    Let $\gamma(s)$ be a non-radial geodesic $\gamma(s)$ on $K$ with $I=1$ and $\|\gamma'(s)\|=1$. By Lemma \ref{lem:two-types}, it can be represented 
    after an appropriate translation of the parameter as 
    \begin{equation}
      \label{eq:nonradial-normalized}
      \gamma(s)
      = q(s) \sqrt{s^2 + 1},
      \qquad s \in (-\infty, +\infty).
    \end{equation}
    where $\|q(s)\| =1$.  
    Differentiating (\ref{eq:nonradial-normalized}) gives
    \begin{equation}\label{eq:nonradial-tangent}
      \gamma'(s)
      = q'(s) \sqrt{s^2 + 1}
      + q(s) \frac{ s}{\sqrt{ s^2 + 1}},
      \quad s \in (-\infty, +\infty).
    \end{equation}
    Since $\langle q(s), q(s) \rangle = 1$ and $\langle q(s), q'(s) \rangle = 0$,  using (\ref{eq:nonradial-normalized}), (\ref{eq:nonradial-tangent}) we obtain
    \begin{equation}\label{eq:gamma-gammaprime}
      \| \gamma(s)\|^2 = s^2+1 ,
      \qquad
    \langle \gamma(s), \gamma'(s) \rangle =  s,
  \end{equation}
    and thus $\langle \gamma(s), \gamma'(s) \rangle = 0$ if and only if $s = 0$.
    At this unique point, $\gamma(s)$ is tangent to the unit sphere $\mathbb{S}^N$.
   
    Next we prove that the curve $\tilde{\gamma}(\tilde{s}), \tilde{s}\in \left(-\frac{\pi}{2}, \frac{\pi}{2}\right)$ defined in (\ref{eq:radialprojection}) is an interval of a geodesic on $\Sigma$ with arc-length parameter.

    Using (\ref{eq:radialprojection}), \eqref{eq:parameter}, (\ref{eq:gamma-gammaprime}), one checks
    \[
    \|\gamma(s)\|=\sqrt{s^2 + 1}=\frac{1}{\cos\tilde s},\qquad
    \frac{ds}{d\tilde s}=\frac{1}{\cos^2\tilde s},
    \]
    and
    \begin{equation}\label{eq:tildeqs}
    \tilde\gamma(\tilde s)=\cos\tilde s\; 
    \gamma \left(\tan\tilde s\right).
    \end{equation}
    Differentiating \eqref{eq:tildeqs} with respect to $\tilde s$ twice gives
    \begin{equation}\label{eq:qprim}
    \tilde\gamma'(\tilde s)
    = -\sin\tilde s\;
    \gamma\left(\tan\tilde s\right)
    + \frac{1}{\cos\tilde s}\;
    \gamma'\left(\tan\tilde s\right),
    \end{equation}
    $$
    \tilde{\gamma}''(\tilde s)
    =
    -\cos \tilde s
    \;
    \gamma\left(\tan\tilde s\right)
    -
    \frac{\sin \tilde{s}}{ \cos^2\tilde s}\,
    \gamma'\left(\tan\tilde s\right)
    +\frac{\sin \tilde{s}}{ \cos^2\tilde s}\,
    \gamma'\left(\tan\tilde s\right)
    + \frac{1}{\cos^3\tilde s}\,
    \gamma''\left(\tan\tilde s\right)
    $$
    \begin{equation}\label{eq:qprimeprime}
    = -\cos\tilde s\;\gamma\left(\tan\tilde s\right)
    + \frac{1}{\cos^3\tilde s}\;
    \gamma''\left(\tan\tilde s\right).
    \end{equation}
    The first term in \eqref{eq:qprimeprime} is proportional to $\tilde\gamma(\tilde s)$ and therefore orthogonal to the tangent space $T_{\tilde\gamma(\tilde s)}\Sigma$. 
    For the second term, since $\gamma$ is a geodesic on $K$, $ \gamma''\left(\tan\tilde s\right)$ is orthogonal to $T_{\gamma(\tan \tilde s)}K =T_{\tilde \gamma(\tilde s)}K $, and hence is orthogonal to $T_{\tilde \gamma(\tilde s)}\Sigma \subset T_{\tilde \gamma(\tilde s)}K$.
    Thus  $\tilde\gamma''(\tilde s)$ is orthogonal to  
    $T_{\tilde \gamma(\tilde s)}\Sigma $, and $\tilde\gamma$ is a geodesic on $\Sigma $.
    
    It follows that $\|\tilde\gamma'(\tilde s)\|$ is constant. Evaluating \eqref{eq:qprim} at $\tilde s=0$ (where $s=0$ by \eqref{eq:parameter}) yields
    \[
    \|\tilde\gamma'(\tilde s)\|=\|\tilde\gamma'(0)\|
    =\|\gamma'(0)\|=1.
    \]
    Therefore $\tilde{s}$ is the arc-length parameter of $\tilde\gamma$ and 
    $$
    \int_{-\pi/2}^{\pi/2}\|\tilde\gamma'(\tilde s)\|\,d\tilde s
    =\pi.
    $$
    
    2) To prove that the curve $\gamma(s)$ defined by \eqref{eq:lift} is a geodesic on $K$, 
it suffices to show that $\gamma''(s)$ is orthogonal to $T_{\gamma(s)}K$.

By \eqref{eq:lift} we have
\begin{equation}\label{eq:lift1}
  \gamma(s) = \tilde{\gamma}(\arctan s) \sqrt{s^2+1}.
\end{equation}
Let us compute $\gamma''(s)$:
\begin{equation}\label{eq:gpsigma}
\gamma'(s) = \tilde{\gamma}'(\arctan s )\frac{1}{\sqrt{s ^2 +1}}
+\tilde{\gamma}(\arctan s)\frac{s}{\sqrt{s ^2 +1}},
\end{equation}
$$
\gamma''(s) =  \frac{\tilde{\gamma}''(\arctan s )}{(s ^2 +1)^{3/2}}
+\frac{(-s) \tilde{\gamma}'(\arctan s )}{(s ^2 +1)^{3/2}}
+
\frac{s \tilde{\gamma}'(\arctan s)}{(s ^2 +1)^{3/2}} 
+\frac{\tilde{\gamma}(\arctan s)}{(s ^2 +1)^{3/2}}
$$
\begin{equation}\label{eq:gppsigma}
 = \frac{\tilde{\gamma}(\arctan s )+\tilde{\gamma}''(\arctan s)}{(s ^2 +1)^{3/2}}.
\end{equation}

The tangent space $T_{\gamma(s)}K$ splits as a direct sum
\[
  T_{\gamma(s)}K \;=\; V_1 \oplus V_2,
\]
where
\[
  V_1 \;=\; T_{\gamma(s)}
  \left(K \cap \mathbb{S}^N (\|\gamma(s)\|) \right),
  \qquad
  V_2 \;=\; \mathrm{span}\{\gamma(s)\}.
\]
Here $\mathbb{S}^N (\|\gamma(s)\|)$ denotes the sphere centered at $O$
with radius $\|\gamma(s)\|$.
Note that $V_1$ is naturally identified with 
$T_{\tilde{\gamma}(\tilde{s})}\Sigma$ via the radial projection.

Since $\tilde{\gamma}$ is a geodesic on $\Sigma$, 
$\tilde{\gamma}''(\arctan s)$ is orthogonal to $T_{\tilde{\gamma}(\tilde{s})}\Sigma$, hence to $V_1$.
Together with the fact that $\tilde{\gamma}(\arctan s)$ is also orthogonal to $V_1$, 
\eqref{eq:gppsigma} implies that $\gamma''(s)$ is orthogonal to $V_1$.

It remains to check that $\gamma''(s)$ is orthogonal to $V_2$; 
for this, we compute $\langle \gamma''(s),\gamma(s)\rangle$ and show that it vanishes.

Since $\langle \tilde\gamma(\tilde s),\tilde\gamma(\tilde s)\rangle = 1$, we have
\[
\left\langle \tilde\gamma'(\tilde s),\tilde\gamma'(\tilde s)\right\rangle
+
\left\langle \tilde\gamma''(\tilde s),\tilde\gamma(\tilde s)\right\rangle
=0,
\]
hence
\begin{equation}\label{eq:gammagammapp}
\left\langle \tilde\gamma''(\tilde s),\tilde\gamma(\tilde s)\right\rangle
=
-\left\langle \tilde\gamma'(\tilde s),\tilde\gamma'(\tilde s)\right\rangle
= -1.
\end{equation}
Using \eqref{eq:lift1}, \eqref{eq:gppsigma} and \eqref{eq:gammagammapp},
$$
\langle \gamma''(s),\gamma(s)\rangle
=
\left\langle
\frac{\tilde{\gamma}(\arctan s )+\tilde{\gamma}''(\arctan s)}{(s ^2 +1)^{3/2}},
\tilde{\gamma}(\arctan s) \sqrt{s^2+1}
\right\rangle
$$
$$
=\frac{
  \left\langle
\tilde{\gamma}(\tilde s )
,
\tilde{\gamma}(\tilde s) 
\right\rangle
+
\left\langle
\tilde{\gamma}''(\tilde s),
\tilde{\gamma}(\tilde s)
\right\rangle
}{s^2+1}
=\frac{1+(-1)}{s^2+1}=0.
$$

Thus $\gamma''(s)$ is orthogonal to $T_{\gamma(s)}K$, and therefore $\gamma(s)$
is a geodesic on $K$.
From (\ref{eq:gpsigma}) it follows that 
$$
\|\gamma'(s)\|= \|\gamma'(0)\|= \|\tilde{\gamma}'(0)\|=1.
$$

Finally, it is immediate from \eqref{eq:lift} that 
$\gamma(s)$ touches $\mathbb{S}^N$ at $\tilde{\gamma}(0)=\gamma(0)$.

3)
Since $\Sigma$ is a closed Riemannian manifold, it is geodesically complete. By 1) the curve $\tilde{\gamma}(\tilde{s})$ is an interval of geodesic on $\Sigma$,
hence the follow limits hold
\[
\lim_{\tilde s \to \frac{\pi}{2}} \tilde{\gamma}(\tilde{s}) = \tilde{\gamma}\left(\frac{\pi}{2}\right) , 
\quad
\lim_{\tilde s \to - \frac{\pi}{2}} \tilde{\gamma}(\tilde{s}) = \tilde{\gamma}\left(-\frac{\pi}{2}\right) .
\]
From (\ref{eq:radialprojection}), \eqref{eq:parameter}, we have
\[
 \lim_{s \to + \infty} \frac{\gamma(s)}{\|\gamma(s)\|}
= \lim_{\tilde s \to \frac{\pi}{2}} \tilde{\gamma}(\tilde{s}) = \tilde{\gamma}\left(\frac{\pi}{2}\right) ,
\quad
\lim_{s \to - \infty} \frac{\gamma(s)}{\|\gamma(s)\|}
= \lim_{\tilde s \to - \frac{\pi}{2}} \tilde{\gamma}(\tilde{s}) = \tilde{\gamma}\left(-\frac{\pi}{2}\right) .
\]

Next we show that
\[
\lim_{s \to +\infty} 
\gamma'(s)
= \tilde{\gamma}\left(\frac{\pi}{2}\right), 
\qquad
\lim_{s \to -\infty} 
{\gamma'(s)}
= -\tilde{\gamma}\left(-\frac{\pi}{2}\right).
\]
Using  $\|\gamma'(s)\|=1$ and (\ref{eq:gamma-gammaprime}) let us compute
$$
\left\| 
\gamma'(s) -  \frac{\gamma(s)}{\|\gamma(s)\|}
\right\|^2
=\|\gamma'(s)\|^2 - 2 
\frac{  \langle \gamma'(s) , \gamma(s)\rangle} {\|\gamma (s)\| }+ \frac{\|\gamma(s)\|^2}{\|\gamma(s)\|^2}
= 2 - 2\frac{ s}{ \sqrt{ s^2 +1}} \to 0 \quad \text{as } s\to +\infty.
$$
Therefore 
$$
\lim_{s \to +\infty} \gamma'(s) 
=
\lim_{s \to +\infty} \frac{\gamma(s)}{\|\gamma(s)\|}= \tilde{\gamma}\left(\frac{\pi}{2}\right). 
$$
Similarly, 
$$
\left\| 
\gamma'(s) +  \frac{\gamma(s)}{\|\gamma(s)\|}
\right\|^2
=\|\gamma'(s)\|^2 + 2 
\frac{  \langle \gamma'(s) , \gamma(s)\rangle} {\|\gamma (s)\| }+ \frac{\|\gamma(s)\|^2}{\|\gamma(s)\|^2}
= 2 + 2\frac{ s}{ \sqrt{ s^2 +1}} \to 0 \quad \text{as } s\to -\infty,
$$
and hence 
$$
\lim_{s \to -\infty} 
\gamma'(s)
= 
- \lim_{s \to -\infty} \frac{\gamma(s)}{\|\gamma(s)\|} = - \tilde{\gamma}\left(-\frac{\pi}{2}\right). 
$$

Lemma~\ref{lem:cone-sphere-correspondence} is proved.
 \end{proof}

\section{Proof of Theorem  \ref{lem:integral} }

Let us define two subsets of  $T^1K$:
\[
  \begin{aligned}
    \mathcal{T}
&= \left\{ (x, v) \in T^1 K  \mid  x \neq k v \text{ for any } k \in \mathbb{R} \right\},\\
\mathcal{S}
&= \left\{ (x, v) \in T^1 K \mid   x=k v \text{ for some } k \in \mathbb{R} \right\}.
  \end{aligned}
\]
Then  $\mathcal{T}\cup\mathcal{S} = T^1 K$ and $\mathcal{T}$ is an open dense subspace of $T^1 K$.
By Lemma \ref{lem:two-types},
any element $(x,v) \in \mathcal{T}$ defines a non-radial geodesic $\gamma(s)$ with $\gamma(0)=x$ and $\gamma'(0)=v$.

In this section, we  prove Theorem  \ref{lem:integral} by first constructing smooth integrals on $\mathcal{T}$ (see Lemma~\ref{lem:smooth-integrals}) 
and then modify them to obtain integrals smooth on 
$T^1K^+ \cup T^1K^-$ and continuous on $T^1K$ (see Lemma~\ref{lem:continuous-integrals}).

\subsection{Construction of first integrals}
For $(x,v)\in\mathcal T$, let $\gamma_{x,v}(s)$ denote the geodesic with $\gamma_{x,v}(0)=x$ and $\gamma'_{x,v}(0)=v$. 
From Lemma~\ref{lem:cone-sphere-correspondence}, it follows that any non-radial geodesic $\gamma_{x,v}(s)$ on $K$ touches the sphere
\[
\mathbb{S}^N(\sqrt{I}) = \{ x \in \mathbb{R}^{N+1} : \|x\|^2= I \}
\]
at a unique point $\gamma_{x,v}(s_0)$, where $s_0=s_0(x,v)$ is a function depends on $(x,v)$ (see Fig.~\ref{fig:geo-s0}). 
Since the cone $K$ is embedded in $\mathbb{R}^{N+1}$,
both the position vector $\gamma_{x,v}(s_0)$ and the velocity vector $\gamma_{x,v}'(s_0)$
can be naturally regarded as vectors in $\mathbb{R}^{N+1}$.
\begin{figure}[htbp]
  \begin{center}
  \includegraphics[scale=0.26]{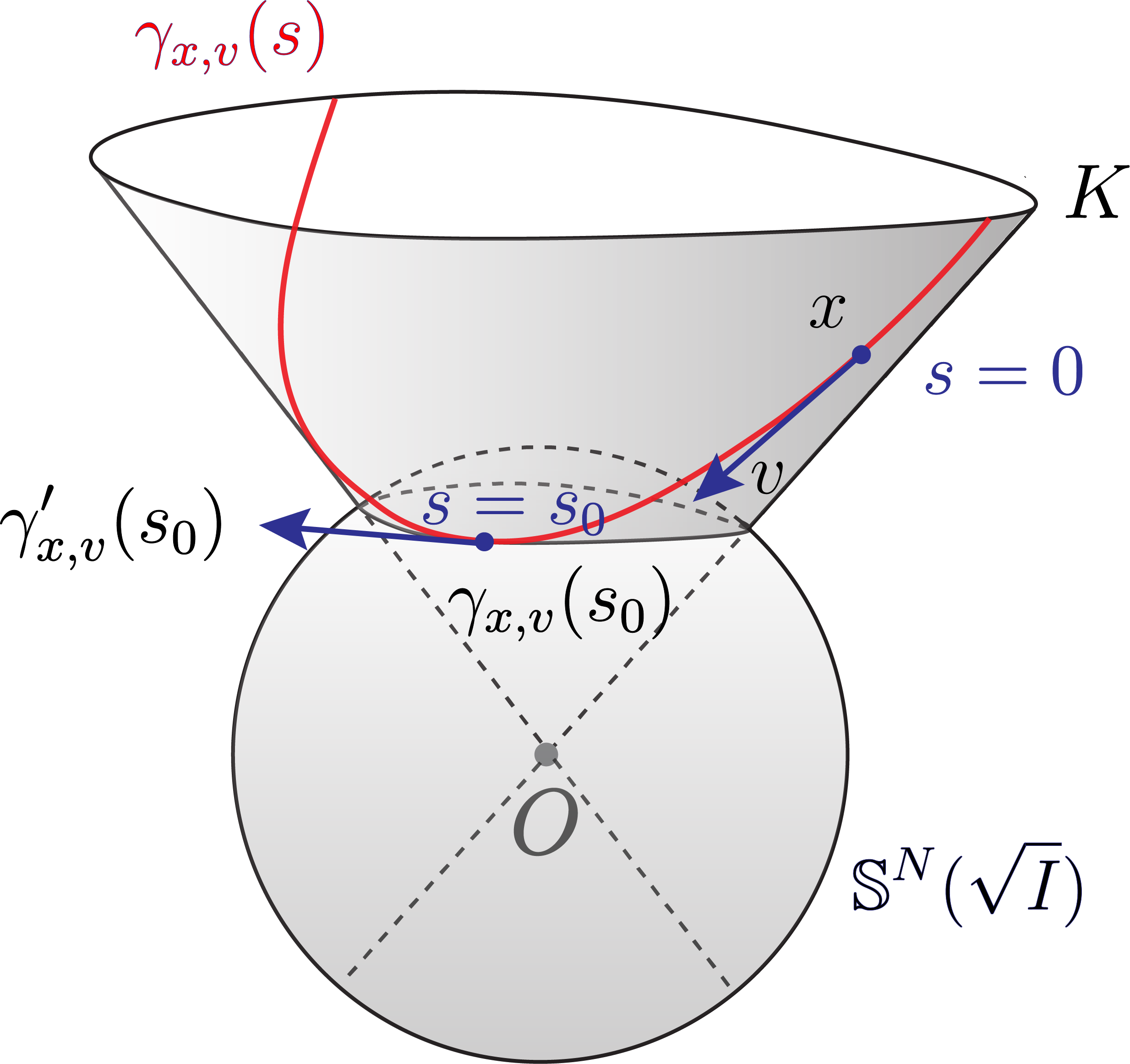}
  \end{center}
  \caption{The geodesic $\gamma_{x,v}(s)$ touches $\mathbb{S}^N(\sqrt{I})$ at the unique point $\gamma_{x,v}(s_0)$.}
  \label{fig:geo-s0}
\end{figure} 
Let us define functions on $\mathcal{T}\subset T^1K$:
\begin{equation}\label{eq:many-integrals1}
  J_j(x,v) := \gamma_{x,v}^j(s_0(x,v)) , \qquad 
  j = 1, \dots, N+1,
\end{equation}
\begin{equation}\label{eq:many-integrals2}
  J_{N+1+j}(x,v) := \gamma_{x,v}^{\prime\, j}(s_0(x,v)) , \qquad 
  j = 1, \dots, N+1.
\end{equation}
Here the superscript $j$ denotes the $j$-th Euclidean coordinate of a vector in $\mathbb{R}^{N+1}$.
By construction, the functions $J_1, \dots, J_{2N+2}$ are constant
along each geodesic trajectory in $\mathcal{T}$ and therefore form first integrals of the geodesic flow restricted to $\mathcal{T}$.

The next lemma establishes the smoothness of integrals $J_1, \dots, J_{2N+2}$ on 
$\mathcal{T}$ and shows that they uniquely determine non-radial geodesics.

\begin{lemma}\label{lem:smooth-integrals}
The functions $J_1, \dots, J_{2N+2}$ are $C^1$-smooth first integrals of the geodesic flow restricted to $\mathcal{T}$.
Moreover, the image point
\[
\mathcal{J} = (J_1, \dots, J_{2N+2}) \in \mathbb{R}^{2N+2}
\]
uniquely determines a non-radial geodesic.
\end{lemma}

\begin{proof}
By Lemma~\ref{lem:two-types}, the open subset \( \mathcal{T} \subset T^1 K^+ \cup T^1 K^- \) is invariant under the geodesic flow \( \rho^s \). 
The geodesic flow restricted to $\mathcal{T}$ corresponds to a $C^1$ vector field on \( T\mathcal{T} \):
$$
X(\rho^s(x,v)) = \partial_s \rho^s(x,v).
$$
By the theory of ordinary differential equations, 
the flow \( \rho^s(x,v) \) depends \( C^1 \)-smoothly on the initial data \( (x,v)\in\mathcal{T} \). 
Moreover, since the vector field \( X \) is \( C^1 \),  \( \rho^s(x,v) \) is \( C^1 \) with respect to the parameter \( s \).

 Define the auxiliary funtion 
  \[
  \phi(x,v,s) := \langle \gamma_{x,v}(s), \gamma'_{x,v}(s) \rangle.
  \]
  From the proof of part 1) of Lemma~\ref{lem:cone-sphere-correspondence}, 
  for each $(x,v)$,  $s_0=s_0(x,v)$ is the unique value of the parameter such that 
  \begin{equation}\label{eq:phixvs0}
  \phi(x,v,s_0)=0.
\end{equation}
  Since $\partial_s \phi(x,v,s)|_{s=s_0} = \|\gamma'_{x,v}(s_0)\|^2 = \|v\|^2 =1 > 0$ and $\phi$ is $C^1$ in $(x,v,s)$, the implicit function theorem implies that $s_0(x,v)$, 
  founded from (\ref{eq:phixvs0}),
  is $C^1$ with respect to $(x,v)$ 
  (later we will prove that $s_0(x,v)=-\langle x, v\rangle$, see (\ref{eq:xv})).
  
  By the construction (\ref{eq:many-integrals1}) and (\ref{eq:many-integrals2}), it follows that
  $
  (J_1(x,v), \dots, J_{2N+2}(x,v))
  $
  depend $C^1$-smoothly of $(x,v)$, as they are compositions of two $C^1$ maps. This proves the $C^1$-smoothness of $J_1,\dots,J_{2N+2}$ on $\mathcal T$.
  
  Finally, to see that $\mathcal J = (J_1,\dots,J_{2N+2})$ uniquely determines a non-radial geodesic, note that 
  $\mathcal{J}$ gives 
  the position $x_0=(J_1,\dots,J_{N+1})$ and velocity $v_0=(J_{N+2},\dots,J_{2N+2})$ at the unique parameter value where $\langle \gamma_{x,v}(s), \gamma'_{x,v}(s)\rangle=0$. 
  Since the geodesic initial value problem has a unique solution, this data determines a unique geodesic. 
\end{proof}

Let us observe that the first \(N+1\) integrals \((J_1,\dots,J_{N+1})= \gamma_{x,v}(s_0)\) can be extended continuously to \(\mathcal S\). 
Indeed, when $v$ tends to the direction of $x$, 
the point $\gamma_{x,v}(s_0)$ tends to $O$, because $I$ tends to $0$. In the same time
  \((J_{N+2},\dots,J_{2N+2})=\gamma_{x,v}'(s_0)\) cannot be extended continuously to \(\mathcal S\), because there is no limit of $\gamma_{x,v}'(s_0)$ when $v$ tends to the direction of $x$ (more precisely the limit depends on how $v$  tends to the direction of $x$). 
  To obtain integrals that are continuous on  $T^1 K$ and $C^1$ on $T^1K^+ \cup T^1K^-$, 
let us define 
\[
  \mathcal{I} \colon T^{1}K \to \mathbb{R}^{2N+2}
\]
by
\begin{equation}\label{eq:I-integrals}
  \mathcal{I}(x, v) =
\begin{cases}
e^{-\frac{1}{I^2(x,v)}} \bigl( \gamma_{x,v}(s_0(x,v)), \gamma_{x,v}'(s_0(x,v)) \bigr), & \text{for } (x, v) \in \mathcal{T}, \\[8pt]
0, & \text{for } (x, v) \in \mathcal{S}.
\end{cases}
\end{equation}
Let $I_k$ denote the $j$-th Euclidean coordinate of $  \mathcal{I}(x, v)$.

\begin{lemma}\label{lem:continuous-integrals}
  The functions $I_1, \dots, I_{2N+2}$ defined by \eqref{eq:I-integrals} are first integrals of the geodesic flow on $T^1 K$ with the following properties:
  \begin{enumerate}
      \item[1)] They are continuous on the entire unit tangent bundle $T^1 K$;
      \item[2)] They are $C^{1}$-smooth on the smooth part $T^{1}K^{+} \cup T^{1}K^{-}$ of $T^1 K$;
      \item[3)] The map $\mathcal{I} = (I_1, \dots, I_{2N+2})$ uniquely determines any non-radial geodesic (where $\mathcal{I} \neq 0$), while mapping all radial geodesics to the origin $0 \in \mathbb{R}^{2N+2}$.
  \end{enumerate}
\end{lemma}
The proof of   2) is lengthy and will be given in the next section.
Below, we provide the proofs for  1) and 3).

\begin{proof}[Proof of parts 1) and 3) of Lemma~\ref{lem:continuous-integrals}]
 First we prove 3).  
  By (\ref{eq:I-integrals}),
  $$
I_k= e^{-\frac{1}{I^2}} J_k, \quad \text{on } \mathcal{T}.
  $$ 
  Since $I$ and $J_k$ are first integrals of the geodesic flow on $\mathcal{T}$, any function of them is also a first integral. Thus, $I_k$ is constant along non-radial geodesics.
  For radial geodesics (i.e., on $\mathcal{S}$), $I_k$ is identically zero.
  Therefore, the functions $I_k$ are first integrals on $T^1 K$. 
  Furthermore, the condition $I_k=0$ on $\mathcal{S}$ implies that all radial geodesics are mapped by $\mathcal{I}$ to the origin $0 \in \mathbb{R}^{2N+2}$.

  Given a value $\mathcal{I} = (I_1, \dots, I_{2N+2}) \neq 0$ on $\mathcal{T}$, one can uniquely recover the values of the integrals $J_k$. Indeed, consider the sum of the square of the first $N+1$ components:
    \[
      \sum_{k=1}^{N+1} (I_k)^2 = \sum_{k=1}^{N+1} \left( e^{-\frac{1}{I^2}} J_k \right)^2 = e^{-\frac{2}{I^2}} \sum_{k=1}^{N+1} (J_k)^2= I^2 e^{-\frac{2}{I^2}}.
\]
Consider the function $f(t) = t e^{-2/t}$ for $t > 0$ (where $t=I^2$). Its derivative is $f'(t) = e^{-2/t}(1 + 2/t)$, which is strictly positive. Thus, $f$ is strictly increasing and bijective onto its image. Consequently, the value $f(I^2)$ uniquely determines $I^2$ (and hence $I$). Once $I$ is determined, the values of $J_k$ are recovered by
\[
J_k = e^{\frac{1}{I^2}} I_k, \quad k = 1,\dots,2N+2.
\]
  By Lemma~\ref{lem:smooth-integrals}, the value of  $\mathcal{J}=(J_1,\dots,J_{2N+2})$ uniquely determines a non-radial geodesic.  
    Hence $\mathcal{I}$ uniquely determines a non-radial geodesic whenever $\mathcal{I} \neq 0\in \mathbb{R}^{2N+2}$.  
  The assertion 3) is proved.

Next we prove the assertion 1).
    Let $(x_n,v_n) \in \mathcal T$ be a sequence converging to a point $(x_0,v_0) \in \mathcal S$.  
    To prove the continuity of $\mathcal{I}$ on $\mathcal{S}$ we will show that 
    $$
I_k(x_n, v_n)\to 0 = I_k(x_0, v_0) 
\quad \text{as~} n \to + \infty,
\quad k=1,\dots, 2N+2.
    $$
    Observe that
    \begin{equation}\label{eq:J-bounds}
      \|(J_1(x,v), \dots, J_{N+1}(x,v))\| \le \|x\|, \qquad
      \|(J_{N+2}(x,v), \dots, J_{2N+2}(x,v))\| = \|v\|.
    \end{equation}
    Hence, 
    \[
    |I_k(x_n,v_n)| =e^{-\frac{1}{I^2(x_n,v_n)}}\, |J_k(x_n,v_n)| \le e^{-\frac{1}{I^2(x_n,v_n)}}\, \left(
      \|x_n\| +  \|v_n\|
     \right).
    \]
    Evaluating (\ref{eq:distance-integral}) in Lemma~\ref{lem:distance} at $s=0$ and using $\|v\|= 1$, we obtain the explicit formula
\begin{equation}\label{eq:explicit-i}
  I(x,v) = \langle x,x \rangle - \langle x,v \rangle^2,
\end{equation}
which implies the continuity of $I$ on $T^1K$. By Lemma~\ref{lem:two-types}, we have $I(x_0, v_0)=0$; thus, continuity implies that $I(x_n, v_n) \to 0$ as $(x_n, v_n) \to (x_0, v_0)$.
Since $e^{-\frac{1}{I^2}} \to  0$ as $I \to 0$, and the term $(\|x_n\| + \|v_n\|)$ is bounded, it follows that
\[
  e^{-\frac{1}{I(x_n,v_n)^2}}\,\bigl( \|x_n\| + \|v_n\| \bigr) \to 0 \quad \text{as } n \to \infty.
\]
Therefore 
    \[
    I_k(x_n,v_n) \to 0 = I_k(x_0,v_0).
    \]
    Thus each $I_k$ is continuous on $\mathcal S$.    
   The assertion 1) is proved.
\end{proof}

\subsection{Proof of smoothness of the first integrals in Theorem \ref{lem:integral}}
In this section, we will prove  part 2)  of Lemma~\ref{lem:continuous-integrals}, i.e., that the map 
$\mathcal{I} \colon T^{1}K \to \mathbb{R}^{2N+2}$ defined by (\ref{eq:I-integrals}) is $C^{1}$-smooth on 
$T^{1}K^{+} \cup T^{1}K^{-}$. Since $T^{1}K^{+}$ and $T^{1}K^{-}$ are disjoint open sets, and they are isometric via the map $p \mapsto -p$, it suffices to prove that $\mathcal{I}$ is $C^{1}$-smooth on $T^{1}K^{+}$.

This result constitutes the most technically involved part of Lemma~5 (equivalently, Theorem~1).
In order to establish the \(C^{1}\)-smoothness of \(\mathcal{I}\), we show that the norm of the Riemann curvature tensor on \(K^{+}\) is of order
\(O(t^{-2})\), and the norm of second fundamental form is of order \(O(t^{-1})\) as
\(t \to 0\) (that is, as the point \(tp\) approaches \(O\)); see Lemma~7.
In Lemma~8, we establish estimates for Jacobi fields along geodesics in \(K^{+}\).
With these preparatory results, 
 Lemma~10 provides a detailed estimation of the behavior of \(d\mathcal{I}\) in a
neighborhood of \(\mathcal{S}\cup T^{1}K^{+}\). 

Following the notation of \cite{ON}, the cone \(K^{+}\) is isometric to the product manifold \((0,\infty)\times\Sigma\) equipped with the Riemannian metric
\[
g = dt^{2} + t^{2} g_{\scriptscriptstyle \Sigma},
\]
where \(g_{\scriptscriptstyle \Sigma}\) denotes the Riemannian metric on \(\Sigma\), and \(t\) is the coordinate on \((0,\infty)\).
Let $R$ denote the Riemann curvature tensor of the manifold $K^{+} \cong (0,\infty) \times \Sigma$.
  For any point $(t, p) \in K^+$, $t\in(0,\infty)$,
  $p\in\Sigma$,
   we decompose the tangent space $T_{(t,p)}K^+$ into the radial subspace $\mathcal{V} = \operatorname{span}\{\partial_t\}$ and the horizontal subspace $\mathcal{H}_{(t,p)}$ (which is tangent to the slice $\{t\} \times \Sigma$):
  $$
T_{(t,p)}K^+ = \mathcal{V}\oplus \mathcal{H}_{(t,p)}.
  $$
  Let 
  \[
  \pi: (0, \infty) \times \Sigma \to \Sigma
  \] 
  be the natural projection onto the second factor, $\pi(t,p) = p$.
  Then we have the linear isomorphism
  $$
  d\pi|_{\mathcal{H}_{(t,p)}} : \mathcal{H}_{(t,p)} \to T_p \Sigma.
  $$
  We denote its inverse map (the \textit{horizontal lift}) by 
  \[
  L_{(t,p)}: T_p\Sigma \to \mathcal{H}_{(t,p)}.
  \]
  Let $\|\cdot\|$ and $\|\cdot\|_{\Sigma}$ denote the norms with respect to $g$ and $g_{\scriptscriptstyle \Sigma}$ respectively. 
  For any horizontal vector 
  $u \in \mathcal{H}_{(t,p)}$, we have the scaling relation
\begin{equation}\label{eq:scale-by-t}
  \|u\| = t \, \|d\pi(u)\|_{\Sigma},
\end{equation}
or equivalently, for any vector $v \in T_p \Sigma$, 
\begin{equation}\label{eq:scale-by-t1}
  \|  L_{(t,p)}(v)\| = t \, \|v\|_{\Sigma}.
\end{equation}
  Let $R_{\Sigma}$ denote the Riemann curvature tensor of $(\Sigma, g_{\scriptscriptstyle \Sigma})$. 
  The \textit{lifted curvature tensor} $R_{(t,\Sigma)}$ on the horizontal subspace $\mathcal{H}_{(t,p)}$ is defined by
\begin{equation}\label{eq:lift-by-t}
  R_{(t,\Sigma)}(u, v)w := L_{(t,p)}\Big( R_{\Sigma}\big( d\pi(u), d\pi(v) \big) d\pi(w) \Big) \in \mathcal{H}_{(t,p)}
\end{equation}
for $u,v,w \in \mathcal{H}_{(t,p)}$ (see \cite{ON}).

\medskip

The next lemma is an immediate corollary of Prop. 7.42 of \cite{ON}.
\begin{lemma}\label{lem:cone-curvature-formula}
  \begin{enumerate}
      \item 
      If any of the vectors $u, v, w \in T_{(t,p)}K$ belongs to $\mathcal{V}$, then
      \[
      R(u, v) w = 0.
      \]
      
      \item 
      If $u, v, w \in \mathcal{H}_{(t,p)}$, then
      \[
      R(u, v) w = R_{(t,\Sigma)}(u, v)w - \frac{1}{t^2} \big( \langle u, w \rangle v - \langle v, w \rangle u \big).
      \]
  \end{enumerate}
\end{lemma}
Indeed, 
  note that $(K^+,g)$ corresponds to the warped product $B\times_f F$, with base $B=\big((0,\infty), dt^2\big)$, fiber $F=(\Sigma, g_{\scriptscriptstyle \Sigma})$, and warping function $f(t)=t$ (see \cite[Definition 7.33]{ON}).

  Since the base is $1$-dimensional, by the skew-symmetry of the first two entries of the Riemann curvature tensor, we have
  $$
  R(\partial_t, \partial_t) \partial_t = 0, 
  \quad 
  R(\partial_t, \partial_t) u = 0, 
  \quad 
  \forall u \in \mathcal{H}_{(t,p)}.
  $$
  Substituting $f(t)=t$, $X=Y= \partial_t$, and $V= u\in \mathcal{H}_{(t,p)}$ into (2) of \cite[Proposition 7.42]{ON}, we have 
  $$
  R(u, \partial_t) \partial_t = \frac{H^f(\partial_t, \partial_t)}{f} u = 0,
  $$
where $H^f$ denotes the Hessian tensor of $f$. In our case, this term vanishes because
  $$H^f(\partial_t, \partial_t) = \partial_t (\partial_t f)- (\nabla_{\partial_t} \partial_t)f =0.
  $$
  Next, substituting $f(t)=t$, $X= \partial_t$, $V=v \in \mathcal{H}_{(t,p)}$, and $W=w \in \mathcal{H}_{(t,p)}$ into (4) of \cite[Proposition 7.42]{ON}, we obtain 
  $$
  R(\partial_t, v) w 
  = \frac{\langle v,w \rangle}{f} \nabla_{\partial_t}(\text{grad} f) 
  = \frac{\langle v,w \rangle}{f} \nabla_{\partial_t}\partial_t
  = 0.
  $$
  Moreover, (3) of \cite[Proposition 7.42]{ON} gives
  $$
  R(u, v) \partial_t = 0, 
  \quad 
  \forall u,v \in \mathcal{H}_{(t,p)}.
  $$
  Therefore, the first assertion of Lemma~\ref{lem:cone-curvature-formula} is verified.

  To prove the second assertion of Lemma~\ref{lem:cone-curvature-formula}, using (5) of \cite[Proposition 7.42]{ON}, we have
  \[
    \begin{aligned}
      R(u,v)w &= R_{(t,\Sigma)}(u,v)w - \frac{\langle \text{grad} f, \text{grad} f \rangle}{f^2} \big( \langle u,w \rangle v - \langle v,w \rangle u \big)\\
      & =
      R_{(t,\Sigma)}(u, v)w - \frac{1}{t^2} \big( \langle u, w\rangle v - \langle v, w\rangle u \big).
    \end{aligned}
  \]
  Here we have used $\langle \text{grad} f, \text{grad} f \rangle = \langle \partial_t, \partial_t \rangle = 1$.

\begin{lemma}\label{lem:cone-curvature}
  There exists a constant $C>0$ (depending only on $\Sigma$ and its embedding in $\mathbb{S}^{N} \subset \mathbb{R}^{N+1}$) such that for every $(t,p)\in K^+$ and all tangent vectors $u,v,w\in T_{(t,p)}K^+$, the following estimates hold:
  \begin{itemize}
    \item[1)] The Riemann curvature tensor of $K^+$ satisfies:
    \[
    \|R(u,v)w\| \le \frac{C}{t^{2}} \, \|u\| \, \|v\| \, \|w\|.
    \]
    \item[2)] The second fundamental form of the embedding $K^+ \hookrightarrow \mathbb{R}^{N+1}$ satisfies:
    \[
      \|\text{II}(u, v)\| \le \frac{C}{t} \, \|u\| \, \|v\|.
    \]
  \end{itemize}
\end{lemma}

\begin{proof}
  We decompose each vector $u, v, w $ into its radial and horizontal components:
  \[
  u = u^r \partial_t + u^h, \quad v = v^r \partial_t + v^h, \quad w = w^r \partial_t + w^h,
  \]
  where $u^r, v^r, w^r \in \mathbb{R}$ and $u^h, v^h, w^h \in \mathcal{H}_{(t,p)}$. Since the decomposition is orthogonal, we have $\|u^h\| \le \|u\|$, $\|v^h\| \le \|v\|$, and $\|w^h\| \le \|w\|$.

  \medskip
  \textbf{1) Estimate for the Riemann curvature tensor.}

  By part 1) of Lemma \ref{lem:cone-curvature-formula} and the multilinearity of the curvature tensor, we have 
  $$
  R(u, v) w = R(u^h, v^h) w^h.
  $$
  Using part 2) of Lemma \ref{lem:cone-curvature-formula}, we estimate the norm:
  \begin{equation}\label{eq:ineq3}
    \|R(u,v)w\| =\|R(u^h,v^h)w^h\|  \le \|R_{(t,\Sigma)}(u^h, v^h)w^h\| + \frac{1}{t^2} \left\| \langle u^h, w^h \rangle v^h - \langle v^h, w^h\rangle u^h \right\|.
  \end{equation}
  Using the triangle inequality and the Cauchy-Schwarz inequality, the second term satisfies:
  \begin{equation}\label{eq:est1}
    \begin{aligned}
      \frac{1}{t^2} \left\| \langle u^h, w^h\rangle v^h - \langle v^h, w^h\rangle u^h \right\|
      &\le \frac{1}{t^2} \left( \|u^h\|\|w^h\|\|v^h\| + \|v^h\|\|w^h\|\|u^h\|\right) \\
      &= \frac{2}{t^2} \, \|u^h\|\,\|v^h\|\,\|w^h\|.
    \end{aligned}
    \end{equation}
  For the first term, by the definition of the lift (\ref{eq:lift-by-t}) and the metric scaling (\ref{eq:scale-by-t1}), we have:
  \begin{equation}\label{eq:eq1}
    \|R_{(t,\Sigma)}(u^h, v^h)w^h\| = 
    \|L_{(t,p)}\left( R_{\Sigma}(d\pi u^h, d\pi v^h)d\pi w^h \right)\|
    =t \cdot \|R_{\Sigma}(d\pi u^h, d\pi v^h)d\pi w^h\|_{\Sigma}.   
  \end{equation}
  Since $\Sigma$ is compact, we can define 
  \[
    C_{\Sigma} = 
    \sup_{
        p\in \Sigma,\;
        x,y,z\in T_p^1\Sigma
    }
    \|R_{\Sigma}(x,y)z\|_{\Sigma}.
  \] 
  Using (\ref{eq:scale-by-t}), (\ref{eq:eq1}) we obtain:
  \begin{equation}\label{eq:est2}
      \|R_{(t,\Sigma)}(u^h, v^h)w^h\| \le t \cdot C_{\Sigma} \left( \frac{\|u^h\|}{t} \right) \left( \frac{\|v^h\|}{t} \right) \left( \frac{\|w^h\|}{t} \right) = \frac{C_{\Sigma}}{t^2} \, \|u^h\|\,\|v^h\|\,\|w^h\|.
  \end{equation}
  Combining (\ref{eq:ineq3}), (\ref{eq:est1}), and (\ref{eq:est2}), we get:
  \[
  \|R(u,v)w\| \le \frac{C_{\Sigma} + 2}{t^2} \, \|u^h\|\,\|v^h\|\,\|w^h\| \le \frac{C_{\Sigma} + 2}{t^2} \, \|u\|\,\|v\|\,\|w\|.
  \]

\medskip

\textbf{2) Estimate for the second fundamental form.}
   We choose local coordinates $z=(z^1, \dots, z^n)$ on an open set 
  $\mathcal{U} \subset \mathbb{R}^n$ and a parametrization 
\[
\psi: \mathcal{U} \to \Sigma \hookrightarrow \mathbb{R}^{N+1}.
\]
Let us consider the local coordinates $(t, z)$ on the cone $K^+$ and the parametrization
\[
  \Phi(t, z) := t \psi(z) : (0, \infty)\times  \mathcal{U}  \to K^+ \hookrightarrow \mathbb{R}^{N+1}.
\]
Let
$$
E_0 = E_0(t,z) = \frac{\partial \Phi(t, z)}{\partial t} 
= \psi(z), 
\quad 
E_i = E_i(t,z) = \frac{\partial \Phi(t, z)}{\partial z^i} 
= t \frac{\partial \psi(z)}{\partial z^i}, 
\quad 
i=1,\dots,n 
$$
be a local frame of $T K^+$.

Let $\bar{\nabla}$ denote the  connection in the ambient Euclidean space $\mathbb{R}^{N+1}$. 
The second fundamental form of $K^+$ as a submanifold of $\mathbb{R}^{N+1}$ is defined as 
$$
\text{II}(X, Y) := 
\left( \bar{\nabla}_X Y\right)^{\perp},
$$
where $\left( \bar{\nabla}_X Y \right)^{\perp}$ denotes the orthogonal projection of $\bar{\nabla}_X Y$ onto the 
normal bundle $NK^+$.

We now compute $\text{II}(E_\alpha, E_\beta)$ at the point $(t,p)=\Psi(t,z) \in K^+$. 
First, since 
$$
  \bar{\nabla}_{E_0} E_0 = \frac{\partial^2 \Phi(t, z)}{\partial t^2} = 0,
$$
we have
\begin{equation}\label{eq:e00}
  \text{II}(E_0, E_0) = 0.
\end{equation}
Next, since 
\[
  \bar{\nabla}_{E_0} E_j 
  = \frac{\partial^2 \Phi(t, z)} {\partial t \partial z^j}
  = \frac{\partial \psi(z)}{\partial z^j}
  = \frac{1}{t} E_j,
\]
which is tangent to $K^+$,
we have
\begin{equation}\label{eq:e0j}
  \text{II}(E_0, E_j) = 0,
  \quad 
  j=1,\dots,n.
\end{equation}

Finally,
since the normal spaces $N_{(t,p)}K^+$, $N_{(1,p)}K^+$ are naturally identified with subspaces of $\mathbb{R}^{N+1}$, by considering $\text{II}(E_i(t,z), E_j(t,z))$, 
$\text{II}(E_i(1,z), E_j(1,z))$ 
both as vectors in $\mathbb{R}^{N+1}$, we will prove the identity 
\begin{equation}\label{eq:scaling-II}
  \text{II}(E_i(t,z), E_j(t,z)) = t \text{II}(E_i(1,z), E_j(1,z)).
\end{equation}
Indeed,
we have 
$$
\bar{\nabla}_{E_i} E_j  =  \frac{\partial^2 \Phi(t, z)} {\partial z^i \partial z^j}
= t \frac{\partial^2 \psi(z)}{\partial z^i \partial z^j}.
$$
Since the normal space $N_{(t,p)}K^+$ is identical to $N_{(1,p)}K^+$ as subspaces of $\mathbb{R}^{N+1}$, we have
$$
  \text{II}(E_i(t,z), E_j(t,z)) = \left( t \frac{\partial^2 \psi}{\partial z^i \partial z^j}(z) \right)^{\perp}
  = t \left(  \frac{\partial^2 \psi}{\partial z^i \partial z^j}(z) \right)^{\perp} 
  = t \text{II}(E_i(1,z), E_j(1,z)),
$$
which proves (\ref{eq:scaling-II}).

Recall that $v^h, w^h$ denote the horizontal parts of $v,w\in T_{(t,p)}K^+$ respectively. Let us write
$$
v^h = \sum_{i=1}^n v^i E_i(t,z),
\quad 
w^h = \sum_{j=1}^n w^j E_j(t,z).
$$
Using (\ref{eq:e00}), (\ref{eq:e0j}) and (\ref{eq:scaling-II}) we have
$$
  \text{II}(v,w) =\text{II}(v^h,w^h) = \sum_{i,j=1}^n v^iw^j \text{II}(E_i(t,z), E_j(t,z)) = \sum_{i,j=1}^n v^iw^j t \text{II}(E_i(1,z), E_j(1,z))
$$
\begin{equation}\label{eq:relation-t}
= t \text{II}
\left(\sum_{i=1}^n v^iE_i(1,z) , \sum_{j=1}^n w^jE_j(1,z)  \right).
\end{equation}
Since $T^1K^+|_{\Sigma}$ is compact, we can define 
\[
  C_{\text{II}} := 
  \sup_{
      p\in \Sigma,\; v_1,v_2\in T_p^1 K^+
  }
  \|\text{II}(v_1,v_2)\|.
\] 
Then we have 
\begin{equation}\label{eq:relation-t-1}
 \left\|\text{II}\left(\sum_{i=1}^n v^iE_i(1,z) , \sum_{j=1}^n w^jE_j(1,z) \right)\right\| 
 \leq 
C_{\text{II}} \left\|\sum_{i=1}^n v^iE_i(1,z) \right\|   \left\|\sum_{j=1}^n w^jE_j(1,z)\right\| .
\end{equation}
Since $E_i(1,z)= \frac{E_i(t,z)}{t}$ as vectors in $\mathbb{R}^{N+1}$, we have
\begin{equation}\label{eq:relation-t-2}
  \begin{aligned}  
    \left\|\sum_{i=1}^n v^iE_i(1,z) \right\| &= \left\|\sum_{i=1}^n v^i\frac{E_i(t,z)}{t} \right\|= \frac{\|v^h\|}{t}, \\
    \left\|\sum_{j=1}^n w^jE_j(1,z)\right\| &= \left\|\sum_{j=1}^n w^j\frac{E_j(t,z)}{t} \right\|=\frac{\|w^h\|}{t}.
  \end{aligned}
\end{equation}
Combining (\ref{eq:relation-t}),  (\ref{eq:relation-t-1}),  (\ref{eq:relation-t-2}) and using $\|v^h\|\leq\|v\|$, $\|w^h\|\leq\|w\|$, we obtain
$$
\|\text{II}(v,w)\|\leq 
t 
C_{\text{II}} \left\|\sum_{i=1}^n v^iE_i(1,z) \right\|   \left\|\sum_{j=1}^n w^jE_j(1,z)\right\| 
= t
C_{\text{II}} \frac{\|v^h\|}{t}\frac{\|w^h\|}{t} \leq \frac{1}{t}
C_{\text{II}} \|v \| \|w\|.
$$

  \medskip

  Finally, by choosing $C = \max\{C_{\Sigma} + 2,  C_{\text{II}}\}$, both estimates 1) and 2) hold.
\end{proof}

\medskip

The following lemma establishes estimates for the Jacobi field along the geodesic $\gamma_{x,v}(s)$ on the cone. 
Let $J(s)$ be the Jacobi field along the geodesic $\gamma_{x,v}(s)$, satisfying
\begin{equation}\label{eq:j-equ}
  \frac{D^2}{ds^2}J(s)
  + R\big(J(s),\gamma'_{x,v}(s)\big)\gamma'_{x,v}(s)=0,
  \end{equation}
with the initial conditions
\begin{equation}\label{eq:j-initial}
  J(0)=\xi,\qquad 
  \frac{D}{ds}J(0)=\eta,
  \qquad 
  \|\xi\|^2+\|\eta\|^2 \neq 0.
\end{equation}

\begin{lemma}\label{lem:cone-jacobi}
  Let $J(s)$ be a Jacobi field along the geodesic $\gamma_{x,v}(s)$ with initial conditions (\ref{eq:j-initial}).
  For any point $(x_0, v_0) \in \mathcal{S}$ with $x_0 \neq O$, there exists an open neighborhood $U$ of $(x_0, v_0)$ in $T^1K$ such that for all $(x, v) \in U \cap \mathcal{T}$, the following estimates hold for all $s \in \mathbb{R}$:
  \begin{align}
    \|J(s)\| &\le \sqrt{\|\xi\|^2 + \|\eta\|^2} \exp\left( \frac{C}{I(x,v)} |s| \right), 
    \label{eq:jacobi-bound-1} \\
    \left\|\frac{D}{ds}J(s)\right\| &\le \sqrt{\|\xi\|^2 + \|\eta\|^2} \exp\left( \frac{C}{I(x,v)} |s| \right). 
    \label{eq:jacobi-bound-2}
  \end{align}
\end{lemma}
\begin{proof}
  Let us use the following notation:
  \begin{equation}\label{eq:def-z}
      Z(s) := \begin{pmatrix} J(s) \\ \frac{D}{ds}J(s) \end{pmatrix},
  \end{equation}
  and define 
 $$
    \|Z(s)\|^2 := \|J(s)\|^2 + \left\|\frac{D}{ds}J(s)\right\|^2.
$$
  Differentiating $\|Z(s)\|$ with respect to $s$ yields
  \begin{equation}\label{eq:dif}
      \frac{d}{ds} \|Z(s)\| = \frac{1}{2\|Z(s)\|} \frac{d}{ds} \|Z(s)\|^2
      = \frac{\langle Z(s), \frac{D}{ds}Z(s) \rangle}{\|Z(s)\|}.
  \end{equation}
  By the Cauchy-Schwarz inequality, we deduce from (\ref{eq:dif}) that
  \begin{equation}\label{eq:est3}
    \left|\frac{d}{ds} \|Z(s)\| \right| 
    = 
    \left|
    \frac{\langle Z(s), \frac{D}{ds}Z(s) \rangle}{\|Z(s)\|}\right| 
    \le \left\|\frac{D}{ds}Z(s)\right\|.
  \end{equation}
  
 Using (\ref{eq:j-equ}) and (\ref{eq:def-z}) we have 
  $$
  \left\|\frac{D}{ds}Z(s)\right\|^2 
  =
  \left\|\frac{D}{ds}J(s)\right\|^2 + 
  \left\|\frac{D^2}{ds^2}J(s)\right\|^2 
  = \left\|\frac{D}{ds}J(s)\right\|^2 + 
  \big\| R\big(J(s),\gamma'_{x,v}(s)\big)\gamma'_{x,v}(s) \big\|^2.
  $$
  Applying Lemma \ref{lem:cone-curvature} and using $t= \|\gamma_{x,v}(s)\|$ we obtain the bound
  \[
  \left\|\frac{D}{ds}Z(s)\right\|^2 \le \left\|\frac{D}{ds}J(s)\right\|^2 + \frac{C^2}{\|\gamma_{x,v}(s)\|^4} \|J(s)\|^2 \| \gamma'_{x,v}(s) \|^4.
  \]
  Since $\|\gamma'_{x,v}(s)\|^2 = \|v\|^2=1$ and $\|\gamma_{x,v}(s)\|^2 \ge I$, this simplifies to
  \[
  \left\|\frac{D}{ds}Z(s)\right\|^2 \le \left\|\frac{D}{ds}J(s)\right\|^2 + \frac{C^2}{I^2} \|J(s)\|^2.
  \]
  
  Fix $(x_0,v_0)\in\mathcal{S}$. Since $I(x,v)$ is a continuous function on $T^1K$ and $I(x_0,v_0)=0$,
  we can choose an open neighbourhood $U$ of $(x_0,v_0)$ such that for all $(x,v)\in U\cap \mathcal{T}$, the function $I$ is sufficiently small to satisfy
  \[
  \frac{C^2}{I^2} \ge 1.
  \]
  Under this condition, we have
 $$
    \left\|\frac{D}{ds}Z(s)\right\|^2 
    \le 
    \frac{C^2}{I^2} \left( \left\|\frac{D}{ds}J(s)\right\|^2 + \|J(s)\|^2 \right) 
    = \left( \frac{C}{I} \right)^2 \|Z(s)\|^2.
 $$
  Taking the square root and combining with (\ref{eq:est3}) gives the inequality
  \begin{equation}\label{eq:est5}
    \left| \frac{d}{ds} \|Z(s)\| \right| 
    \le
     \frac{C}{I} \|Z(s)\|.
  \end{equation}
  From (\ref{eq:est5}) we have
  \begin{equation}\label{eq:bound}
    \|Z(s)\| \le \|Z(0)\| \exp\left( \frac{C}{I} |s| \right), \quad s\in \mathbb{R}.
  \end{equation}
  Indeed, (\ref{eq:est5}) gives
  \[
   - \frac{C}{I} \|Z(s)\|\le \frac{d}{ds} \|Z(s)\| \le \frac{C}{I} \|Z(s)\|\,.
  \]
  For $s>0$, 
  dividing both sides of $\frac{d}{ds} \|Z(s)\| \le \frac{C}{I} \|Z(s)\|$ by $\|Z(s)\|$ (since the Jacobi field is non-zero, $\|Z(s)\| \neq 0$) we obtain 
  \[
      \frac{d}{ds} \log \|Z(s)\| = \frac{1}{\|Z(s)\|} \frac{d}{ds} \|Z(s)\| \le \frac{C}{I}\,.
  \]
  Integrating both sides yields
  \[
      \log \|Z(s)\| - \log \|Z(0)\| = \int_0^s \frac{d}{d\tau} \log \|Z(\tau)\| \, d\tau \le \int_0^s \frac{C}{I} \, d\tau = \frac{C}{I} s\,.
  \]
  Thus we obtain 
  \begin{equation}\label{eq:z-exp}
    \|Z(s)\| \le \|Z(0)\| \exp\left( \frac{C}{I} s \right) = \|Z(0)\| \exp\left( \frac{C}{I} |s| \right).
  \end{equation}
  Similarly, for $s<0$, from (\ref{eq:est5}), we know that 
  $$
  \frac{d}{ds} \|Z(s)\| \ge -\frac{C}{I} \|Z(s)\|,
  $$ 
  which implies 
  $$
  \frac{d}{ds} \log \|Z(s)\| \ge -\frac{C}{I}.
  $$ 
  Integrating from $s$ to $0$ gives
  \[
      \log \|Z(0)\| - \log \|Z(s)\| = \int_s^0 \frac{d}{d\tau} \log \|Z(\tau)\| \, d\tau \ge \int_s^0 \left( -\frac{C}{I} \right) \, d\tau = -\frac{C}{I}(0-s) = \frac{C}{I} s\,.
  \]
  Hence we get
  \[
      \log \|Z(s)\| \le \log \|Z(0)\| - \frac{C}{I} s = \log \|Z(0)\| + \frac{C}{I} |s|\,,
  \]
  which leads to the same bound (\ref{eq:z-exp}) as in the case for $s>0$.

  Recall that $\|Z(0)\| = \sqrt{\|\xi\|^2 + \|\eta\|^2}$. Since
  \[
    \|J(s)\| \le \|Z(s)\| \quad \text{and} \quad \left\|\frac{D}{ds}J(s)\right\| \le \|Z(s)\|,
  \]
  it follows from (\ref{eq:bound}) that the estimates (\ref{eq:jacobi-bound-1}) and (\ref{eq:jacobi-bound-2}) hold.
\end{proof}

Let 
\begin{equation}\label{eq:tauplus}
  \mathcal{S}^+:=  \mathcal{S} \cap T^1K^+,
\qquad 
\mathcal{T}^+:=  \mathcal{T} \cap T^1K^+.
\end{equation}
To establish the $C^1$-smoothness of the map $\mathcal{I}$ on  $ T^1K^+$, we need to verify that the differential of $\mathcal{I}$ behaves continuously near the set $ \mathcal{S}^+$. 
Lemma~\ref{lem:diff-at-singular} shows that $\mathcal{I}$ is differentiable at any point  $(x_0,v_0) \in \mathcal{S}^+$ and that the differential  $d\mathcal{I}: T_{(x_0,v_0)}(T^1K)\to \mathbb{R}^{2N+2}$ is the zero map.
Lemma~\ref{lem:tangent-map-vanish} proves that for $(x,v)\in \mathcal{T}^+$ the differential $d\mathcal{I}: T_{(x,v)}(T^1K)\to \mathbb{R}^{2N+2}$  tends to the zero map as the point $(x,v)$ tends to $(x_0,v_0) \in \mathcal{S}^+$. 

\begin{lemma}\label{lem:diff-at-singular}
  For any point $(x_0, v_0) \in \mathcal{S}^+$, the map $\mathcal{I}$ is differentiable at $(x_0, v_0)$, and its differential is the zero map:
  \[
  d\mathcal{I} = 0.
  \]
\end{lemma}

\begin{proof}
  Let $(x_0, v_0) \in \mathcal{S}^+$. 
  By definition $\mathcal{I}(x_0, v_0) = 0$. 
  Let us consider an arbitrary smooth curve $c(\varepsilon)=(x(\varepsilon),v(\varepsilon)): (-\epsilon_0, \epsilon_0) \to T^1K$ with $c(0) = (x_0, v_0)$ and $c'(0) = (\xi, \eta) \in T_{(x_0,v_0)}(T^1K)$. To prove that $d\mathcal{I}=0$ at $(x_0,v_0)=c(0)$, 
  it suffices to show that
  $$
  \lim_{\varepsilon \to 0} \frac{\| \mathcal{I}(c(\varepsilon)) - \mathcal{I}(c(0)) \|}{|\varepsilon|} = 0.
$$
  i.e.,
$$
    \lim_{\varepsilon \to 0} \frac{\| \mathcal{I}(c(\varepsilon)) \|}{|\varepsilon|} = 0.
$$
  By (\ref{eq:I-integrals}) and (\ref{eq:J-bounds}),
  we have 
  \[
    \| \mathcal{I}(c(\varepsilon)) \| \leq  e^{-\frac{1}{I^2(c(\varepsilon))}} 
    \left(\|x(\varepsilon)\|^2 + \|v(\varepsilon)\|^2 \right)^{\frac{1}{2}} 
    =
    e^{-\frac{1}{I^2(c(\varepsilon))}} \|c(\varepsilon)\|.
    \]
  Since the curve $c(\varepsilon)=(x(\varepsilon),v(\varepsilon)), \varepsilon\in (-\epsilon, \epsilon)$ is bounded, there exists a constant $M > 0$ such that $\|c(\varepsilon)\| \leq M$. Hence 
  \begin{equation}\label{eq:est-M}
  \| \mathcal{I}(c(\varepsilon)) \| \leq M  e^{-\frac{1}{I^2(c(\varepsilon))}}.
\end{equation}
  Since the function $I: T^1K \to \mathbb{R}$ is smooth (as shown in \eqref{eq:explicit-i}), there exists a constant $L > 0$ such that
  \[
  |I(c(\varepsilon))| = |I(c(\varepsilon)) - I(c(0))| \le L |\varepsilon|,
  \]
  for $\varepsilon$ sufficiently close to 0. 
  Hence 
  \begin{equation}\label{eq:est-L}
    \frac{1}{|\varepsilon|} \le \frac{L}{  |I(c(\varepsilon))| }.
  \end{equation}
  Combining (\ref{eq:est-M}) and (\ref{eq:est-L}) we have
  \begin{equation}\label{eq:limit}
    0 \le \frac{\| \mathcal{I}(c(\varepsilon)) \|}{|\varepsilon|} \le M \frac{e^{-\frac{1}{I^2(c(\varepsilon))}}}{|\varepsilon|} \le ML \frac{e^{-\frac{1}{I^2(c(\varepsilon))}}}{|I(c(\varepsilon))|}.
  \end{equation}
  Since  $|I(c(\varepsilon))|\to |I(c(0))|=0$ as $\varepsilon \to 0$, and  $ \frac{1}{y} e^{-1/y^2}\to 0$ as  $y \to 0$, we know 
  \begin{equation}\label{eq:ics}
    \frac{e^{-\frac{1}{I^2(c(\varepsilon))}}}{|I(c(\varepsilon))|} \to 0 
    \quad 
    \text{~as~} \varepsilon\to0.
  \end{equation}
  Thus, by (\ref{eq:limit}) and (\ref{eq:ics}) we obtain 
  $$
  \frac{\| \mathcal{I}(c(\varepsilon)) \|}{|\varepsilon|}  \to 0 
  \quad 
  \text{~as~} \varepsilon\to0,
  $$
 proving that $\mathcal{I}$ is differentiable at $(x_0, v_0)$ with vanishing derivative.

  Lemma~\ref{lem:diff-at-singular} is proved.
\end{proof}

\begin{lemma}\label{lem:tangent-map-vanish}
  Let $(x_0, v_0) \in \mathcal{S}^+$ and $(x,v)\in \mathcal{T}^+$. Then the differential $d\mathcal{I}: T_{(x,v)}(T^1K^+) \to \mathbb{R}^{2N+2}$ tends to the zero map as $(x,v)$ tends to $(x_0, v_0)$. 
\end{lemma}

To prove the Lemma we estimate the norm of the differential $d \mathcal{I}$ to obtain that the norm 
$\|d \mathcal{I}\|$ can be controlled by a function of $I$, 
which tends to $0$ as $I$ tends to $0$ (see (\ref{eq:bound-est1})).

\begin{proof}
  Recall (see, e.g., \cite{Pat}, Section 1.3) that the tangent space $T_{(x,v)}(TK^+)$ at a point $(x,v) \in TK^+$ admits a canonical isomorphism 
\begin{equation}\label{eq:decomposition}
  T_{(x,v)}(TK^+) \cong T_xK^+ \oplus T_xK^+.
\end{equation}
This isomorphism is defined as follows.
Let $\mathcal{Z} \in T_{(x,v)}(TK^+)$ denote a tangent vector. Let $(x(\varepsilon),v(\varepsilon))\subset TK^+$ be a smooth curve adapted to $\mathcal{Z}$, meaning that it satisfies:
$$
(x(0), v(0)) = (x,v) \quad \text{and} \quad \frac{d}{d\varepsilon}\Big|_{\varepsilon=0} (x(\varepsilon), v(\varepsilon)) = \mathcal{Z}.
$$
 The isomorphism (\ref{eq:decomposition}) maps $\mathcal{Z}$ to the pair $(\xi, \eta)$ given by:
    \begin{equation}\label{eq:decomposition1}
      \mathcal{Z} \mapsto (\xi , \eta) := \left( \frac{d}{d\varepsilon}\Big|_{\varepsilon=0} x(\varepsilon), \  \frac{D}{d\varepsilon}\Big|_{\varepsilon=0} v(\varepsilon) \right).
    \end{equation}
    Here $\frac{D}{d\varepsilon}$ denotes the covariant derivative along the curve $x(\varepsilon)\subset K^+$.
    The \textit{Sasaki metric} on $TK^+$ is defined to be the Riemannian metric
    that makes the decomposition~(\ref{eq:decomposition}) orthogonal:
    \[
      \| \mathcal{Z} \|_{\text{Sasaki}}^2 := \| \xi \|^2 + \| \eta \|^2.
    \]

    Under the canonical isomorphism (\ref{eq:decomposition1}), the vector $\mathcal Z$ is tangent to the unit tangent bundle $T^1K^+$ at $(x,v)$
    if and only if (see, e.g., \cite{Pat}, Section 1.3)
    \begin{equation}\label{eq:decomposition1-t1}
      \langle \eta, v \rangle = 0.
    \end{equation}
    Indeed,  the unit tangent bundle $T^1K^+$ is the level set $f^{-1}(1)$ of the smooth function $f(x,v)= \langle v,v\rangle$ defined on $TK^+$, and $1$ is a regular value of $f$. The tangent vector $\mathcal{Z}$ at $(x,v)\in T^1K^+$ is tangent to $T^1K^+$ if and only if 
    $(df)_{(x,v)}(\mathcal{Z})=0$. That is
$$
    \frac{d}{d\varepsilon} \langle v(\varepsilon), v(\varepsilon) \rangle \Big|_{\varepsilon=0}  = 2 \left\langle \frac{D}{d\varepsilon}v(\varepsilon), v(\varepsilon) \right\rangle\Big|_{\varepsilon=0}  = 2 \langle \eta, v \rangle = 0.
$$

    We define the operator norm $\| d\mathcal{I} \|$ with respect to this Sasaki metric restricted to the tangent space of the unit bundle:
    \begin{equation}\label{eq:def-norm-dI}
      \| d \mathcal{I} \|:=
      \sup_{\substack{\mathcal{Z} \in T_{(x,v)}(T^1K^+), \\ 
      \|\mathcal{Z} \|_{\text{Sasaki}}=1}}
      {\| d \mathcal{I}(\mathcal{Z} ) \|_{\mathbb{R}^{2N+2}}}=
      \sup_{\substack{(\xi, \eta) \in T_xK^+ \oplus T_xK^+, \\ \langle \eta, v \rangle = 0, \\ \| \xi \|^2 + \| \eta \|^2=1}} {\| d \mathcal{I}(\xi, \eta) \|_{\mathbb{R}^{2N+2}}}.
    \end{equation}
 
The computation of $d\mathcal{I}$ consists of the following steps.

\textbf{Step 1: Differentiation of the function $s_0(x,v)$.}

By Lemma \ref{lem:two-types} and the conditions
$$
\langle \gamma_{x,v}(s_0), \gamma'_{x,v}(s_0) \rangle = 0, 
\quad
\langle \gamma_{x,v}(s_0), \gamma_{x,v}(s_0) \rangle = I,
$$ 
\begin{equation}\label{eq:initial}
  \gamma_{x,v}(0)=x,\quad\, \gamma'_{x,v}(0)=v, 
\end{equation}
we know that the geodesic $\gamma_{x,v}(s)$ is of the following form
\begin{equation} \label{eq:geo_form}
    \gamma_{x,v}(s) = q(s) \sqrt{(s-s_0)^2 + I},
\end{equation}
where $q(s) \in \Sigma$.
Using (\ref{eq:initial}) and (\ref{eq:geo_form}) we have
$$
    x = q(0) \sqrt{s_0^2 + I},
$$
$$
    v = q(0) \frac{- s_0}{\sqrt{s_0^2 + I}} + q'(0) \sqrt{ s_0^2 + I}.
$$
Recalling that $q(s) \in \Sigma$, we have $\langle q(0), q(0) \rangle = 1$ and $\langle q(0), q'(0) \rangle = 0$. 
Therefore, 
\begin{equation}\label{eq:xv}
    \langle x, v \rangle = \left\langle q(0)\sqrt{s_0^2 + I}, \, q(0) \frac{-s_0}{\sqrt{ s_0^2 + I}} \right\rangle = -s_0.
\end{equation}
Here the inner product is taken in the ambient Euclidean space $\mathbb{R}^{N+1}$.
We compute the derivative
$$
ds_0(\xi, \eta)=  \frac{d}{d\varepsilon}\Big|_{\varepsilon=0} s_0(x(\varepsilon), v(\varepsilon))
= - \frac{d}{d\varepsilon}\Big|_{\varepsilon=0} \langle x(\varepsilon), v(\varepsilon) \rangle 
= - \left( \left\langle \frac{d x}{d\varepsilon}, v \right\rangle + \left\langle x, \frac{d v}{d\varepsilon} \right\rangle \right)\Big|_{\varepsilon=0}
$$
\begin{equation}\label{eq:ds0-1}
 = - \langle \xi, v \rangle - \langle x, \eta \rangle - \langle x, \text{II}(\xi, v) \rangle.
  \end{equation}
Here we use the formula: 
$$
\frac{d v}{d\varepsilon}\Big|_{\varepsilon=0} 
= \frac{D v}{d\varepsilon}\Big|_{\varepsilon=0} + \text{II}(\xi, v) = \eta + \text{II}(\xi, v).
$$
Since $x$ is orthogonal to $\text{II}(\xi, v)\in N_xK^+$, (\ref{eq:ds0-1}) simplifies to 
\begin{equation}\label{eq:ds0}
  ds_0(\xi, \eta) = - \langle \xi, v \rangle - \langle x, \eta \rangle.
\end{equation}

By \eqref{eq:ds0} we have
\[
|ds_0(\xi,\eta)|
\le \|v\|\,\|\xi\| + \|x\|\,\|\eta\|.
\]
Fix \((x_0,v_0)\in \mathcal{S}^+\).
Then there exist an open neighborhood \(U_1\) of \((x_0,v_0)\) in $T^1K^+$ and a constant \(C_1>0\) such that
\[
\|v\| + \|x\| \le C_1
\quad \text{for all } (x,v)\in U_1.
\]
Consequently,
\begin{equation}\label{eq:bound-ds0}
|ds_0(\xi,\eta)|
\le C_1\bigl(\|\xi\| + \|\eta\|\bigr) 
\quad \text{for all } (x,v)\in U_1.
\end{equation}

\textbf{Step 2: Differentiation of the map $(x,v)\mapsto\bigl( \gamma_{x,v}(s_0(x,v)), \gamma_{x,v}'(s_0(x,v)) \bigr)$.}

    We define the Jacobi field $J(s)$ along the geodesic $\gamma_{x,v}(s)$ by the variation of position:
    \begin{equation}\label{eq:def-var-j}
    J(s) := \frac{\partial}{\partial \varepsilon}\Big|_{\varepsilon=0} 
    \gamma_{x(\varepsilon), v(\varepsilon)}(s).
  \end{equation}
Then $J(s)$ satisfies (\ref{eq:j-equ}) with initial conditions (see  \cite{Pat}, Section 1.5):
    \begin{equation}\label{eq:jacobi-xi-eta}
      J(0) = \frac{d}{d\varepsilon}\Big|_{\varepsilon=0} x(\varepsilon) = \xi, \qquad 
      \frac{D}{ds}J(0) = \frac{D}{d\varepsilon}\Big|_{\varepsilon=0} v(\varepsilon) = \eta.
    \end{equation}

    Following (\ref{eq:xv}), let $s_0(\varepsilon):=- \langle x(\varepsilon),v(\varepsilon)\rangle$ and  
    $s_0 = s_0(0) = -\langle x,v\rangle$.
    We now compute the derivative of the map 
    $$
    \varepsilon \mapsto 
    \bigl( 
      \gamma_{x(\varepsilon), v(\varepsilon)}(s_0(\varepsilon))
      , \;
      \gamma_{x(\varepsilon), v(\varepsilon)}'(s_0(\varepsilon))
      \bigr)
    $$
     at $\varepsilon=0$ in the ambient space $\mathbb{R}^{2N+2}$.

For $\varepsilon \mapsto \gamma_{x(\varepsilon), v(\varepsilon)}(s_0(\varepsilon))$, 
using  (\ref{eq:ds0}) and (\ref{eq:def-var-j}),
we get
        \begin{equation}\label{eq:diff-pos}
        \begin{aligned}
          \frac{d}{d\varepsilon}\Big|_{\varepsilon=0} 
          \gamma_{x(\varepsilon), v(\varepsilon)}(s_0(\varepsilon))
          &= 
          \frac{d s_0(\varepsilon)}{d\varepsilon}\Big|_{\varepsilon=0} 
        \cdot 
          \frac{\partial  \gamma_{x(\varepsilon), v(\varepsilon)}(s)}{\partial s}\Big|_{\varepsilon=0, s=s_0} 
          + 
          \frac{\partial  \gamma_{x(\varepsilon), v(\varepsilon)}(s_0)}{\partial \varepsilon}\Big|_{\varepsilon=0} \\
          &= 
          ds_0(\xi, \eta)\, \gamma'_{x,v}(s_0) + J(s_0).
        \end{aligned}
        \end{equation}  
        For 
        $\varepsilon \mapsto  
        \gamma_{x(\varepsilon), v(\varepsilon)}'(s_0(\varepsilon))
        $, 
        using (\ref{eq:def-var-j}) and $\frac{D}{ds}\gamma'_{x,v}(s) = 0$,
        we obtain
        \begin{equation}\label{eq:diff-ves}
          \begin{aligned}
        \frac{d}{d\varepsilon}\Big|_{\varepsilon=0} 
        \gamma_{x(\varepsilon), v(\varepsilon)}'(s_0(\varepsilon))
        =& \frac{d s_0(\varepsilon)}{d\varepsilon}\Big|_{\varepsilon=0} 
        \cdot 
        \frac{\partial  \gamma_{x(\varepsilon), v(\varepsilon)}'(s)}{\partial s}\Big|_{\varepsilon=0, s=s_0}  
        + 
        \frac{\partial  \gamma_{x(\varepsilon), v(\varepsilon)}'(s_0)}{\partial \varepsilon}\Big|_{\varepsilon=0} \\
        =& ds_0(\xi, \eta) \left( \frac{d}{ds}\Big|_{s=s_0} \gamma'_{x,v}(s)  \right)
        + \left( \frac{\partial}{\partial\varepsilon}\Big|_{\varepsilon=0}
        \frac{\partial}{\partial s}\Big|_{s=s_0}  
        \gamma_{x(\varepsilon), v(\varepsilon)}(s)
        \right)\\
        =&  ds_0(\xi, \eta) \left( 
          \frac{D}{ds}\Big|_{s=s_0} \gamma'_{x,v}(s)+ 
          \text{II}( \gamma'_{x,v}(s_0),  \gamma'_{x,v}(s_0))  \right) + 
          \frac{d}{d s}\Big|_{s=s_0}  J(s)  \\
        =&  ds_0(\xi, \eta) \,
          \text{II}(\gamma_{x,v}'(s_0), \gamma_{x,v}'(s_0))+ 
        \left(    \frac{D}{ds}J(s_0) + \text{II}(J(s_0), \gamma_{x,v}'(s_0))\right) .
      \end{aligned}
    \end{equation}

\textbf{Step 3: Differentiation of the function $I(x,v)$.}
Recall that $I(x,v) = \langle x,x \rangle - \langle x,v \rangle^2$, and $s_0= - \langle x, v \rangle$
(see (\ref{eq:xv})).
Using (\ref{eq:ds0}), we compute the derivative
$$
dI(\xi, \eta) = \frac{d}{d\varepsilon}\big|_{\varepsilon=0} I(x(\varepsilon), v(\varepsilon))
= \frac{d}{d\varepsilon}\Big|_{\varepsilon=0} \langle x(\varepsilon), x(\varepsilon) \rangle - \frac{d}{d\varepsilon}\Big|_{\varepsilon=0} \langle x(\varepsilon), v(\varepsilon) \rangle^2
$$
$$
= 2 \left\langle x, \xi  \right\rangle 
+2 \langle x, v \rangle \frac{d}{d\varepsilon}\Big|_{\varepsilon=0} \left( -\langle x(\varepsilon), v(\varepsilon) \rangle \right)
= 2 \langle x, \xi \rangle + 2 \langle x, v \rangle  ds_0(\xi, \eta).
$$
Using (\ref{eq:bound-ds0}), for all $(x,v)\in U_1$ we have 
$$
\begin{aligned}
  |dI(\xi, \eta) |
  & \leq 2\|x\|\|\xi\| + 2\|x\|\|v\|C_1 (\|\xi\| + \|\eta\|) \\
  & \leq 2\|x\|\left(\|\xi\| +\|\eta\|\right) +  2\|x\|\|v\|C_1 \left(\|\xi\| + \|\eta\|\right)\\
  & \leq \left( 2\|x\| +  2\|x\|\|v\|C_1\right) \left(\|\xi\| + \|\eta\|\right).
\end{aligned}
$$
Let $C_2>0$ be an upper bound of $ 2\|x\|+2\|x\|\|v\|C_1 $ on $U_1$. Then 
\begin{equation}\label{eq:bound-dI}
  |dI(\xi, \eta)|  \le C_2 (\|\xi\| + \|\eta\|).
\end{equation}

\textbf{Step 4: The computation of $d\mathcal{I}$ and estimates of $\|d\mathcal{I}\|$.}

We now combine the results from the previous steps to compute  $d\mathcal{I}(\xi, \eta)$.
By (\ref{eq:I-integrals}), we have 
\[
d \mathcal{I}(\xi, \eta) 
=  \frac{d}{d\varepsilon}\Big|_{\varepsilon=0} \mathcal{I}(x(\varepsilon), v(\varepsilon)) 
= \frac{d}{d\varepsilon}\Big|_{\varepsilon=0} e^{-\frac{1}{I^2(x(\varepsilon),v(\varepsilon))}}
\bigl( \gamma_{x(\varepsilon),v(\varepsilon)}
(s_0(\varepsilon)), \gamma'_{x(\varepsilon), v(\varepsilon)}(s_0(\varepsilon)) \bigr).
\]
Therefore, using (\ref{eq:diff-pos}) and (\ref{eq:diff-ves}), we have 
\begin{equation}\label{eq:total-dI}
\begin{aligned}
&d\mathcal{I}(\xi, \eta) = \frac{2 \, dI(\xi, \eta)}{I^3} e^{-\frac{1}{I^2}} \left( \gamma_{x,v}(s_0), \gamma_{x,v}'(s_0)\right) 
+ e^{-\frac{1}{I^2}}\\
& \qquad  \left(
  ds_0(\xi, \eta)\, \gamma'_{x,v}(s_0) + J(s_0), 
  \quad 
  ds_0(\xi, \eta) \left( 
    \text{II}( \gamma_{x,v}'(s_0),  \gamma_{x,v}'(s_0))\right) + 
    \frac{D}{ds}J(s_0)  + 
    \text{II}(J(s_0),  \gamma_{x,v}'(s_0))
  \right).
\end{aligned}
\end{equation}
Using  (\ref{eq:bound-ds0}), (\ref{eq:bound-dI}), 
$\|\gamma_{x,v}(s_0)\| =\sqrt{I}$, 
$\|\gamma_{x,v}'(s_0)\| = 1$,
and part 2) of Lemma~\ref{lem:cone-curvature} (substituting $t= \|\gamma_{x,v}(s_0)\|$), we have
\[
\begin{aligned}
\|d\mathcal{I}(\xi, \eta)\| 
&\le \frac{2 |dI(\xi, \eta)|}{I^3} e^{-\frac{1}{I^2}} 
\left(
  \|\gamma_{x,v}(s_0)\| +  \|\gamma_{x,v}'(s_0)\| \right) 
  + e^{-\frac{1}{I^2}} 
\left( |ds_0(\xi, \eta)|\|\gamma_{x,v}'(s_0)\| 
  + \|J(s_0)\| \right) \\
&\quad + e^{-\frac{1}{I^2}}  
\left(
  |ds_0(\xi, \eta)| \,
\|\text{II}( \gamma_{x,v}'(s_0),  \gamma_{x,v}'(s_0))\|
  + \|\frac{D}{ds}J(s_0)\| 
  + 
 \| \text{II}(J(s_0),  \gamma_{x,v}'(s_0))\|
   \right) \\
&\le \frac{2 |dI(\xi, \eta)|}{I^3} e^{-\frac{1}{I^2}} 
\left(
  \|\gamma_{x,v}(s_0)\| +  \|\gamma_{x,v}'(s_0)\| \right) 
  + e^{-\frac{1}{I^2}} 
\left( |ds_0(\xi, \eta)|\|\gamma_{x,v}'(s_0)\| 
  + \|J(s_0)\| \right) \\
&\quad + e^{-\frac{1}{I^2}}  
\left(
  |ds_0(\xi, \eta)| \frac{C\|\gamma_{x,v}'(s_0)\|^2}{\|\gamma_{x,v}(s_0)\|}  
  + \|\frac{D}{ds}J(s_0)\| 
  + \frac{C\|J(s_0)\|\|\gamma_{x,v}'(s_0)\|}{\|\gamma_{x,v}(s_0)\|}  \right) \\
&\le \frac{2 C_2(\|\xi\|+\|\eta\|)}{I^3} e^{-\frac{1}{I^2}} 
(\sqrt{I} + 1) + e^{-\frac{1}{I^2}}  C_1(\|\xi\|+\|\eta\|) 
+ e^{-\frac{1}{I^2}}  \|J(s_0)\|
\\
&\quad + e^{-\frac{1}{I^2}} 
\left( C_1(\|\xi\|+\|\eta\|)\frac{C}{\sqrt{I}}
+ \|\frac{D}{ds}J(s_0)\|
+ \frac{C}{\sqrt{I}} \|J(s_0)\| \right)
\\
&\le (\|\xi\|+\|\eta\|) e^{-\frac{1}{I^2}} 
\left(\frac{2 C_2}{I^3} 
(\sqrt{I}+1) +  
 C_1(1 + \frac{C}{\sqrt{I}}) \right)
\\
&\quad + e^{-\frac{1}{I^2}} 
(1+\frac{C}{\sqrt{I}}) \left( \|J(s_0)\| + \|\frac{D}{ds}J(s_0)\| \right).
\end{aligned}
\]
By Lemma \ref{lem:cone-jacobi}, there exists an open neighborhood $U_2\subset U_1$ of $(x_0, v_0)$ in $T^1K$ such that for $(x, v) \in U_2 \cap \mathcal{T}^+$,
\[
\|J(s_0)\| + \|\frac{D}{ds}J(s_0)\| \le 2 \sqrt{\|\xi\|^2 + \|\eta\|^2} \exp\left( \frac{C |s_0|}{I} \right).
\]
Let $C_4>0$ be an upper bound of $C|s_0|$ on $U_2$. We have
$$
  \|d\mathcal{I}(\xi, \eta)\| \le 
  \left(\|\xi\|+\|\eta\|\right)e^{-\frac{1}{I^2}}
  \left(
\frac{2 C_2}{I^3} (\sqrt{I}+1)
+ C_1(1 + \frac{C}{\sqrt{I}}) \right) 
+ 2\sqrt{\|\xi\|^2 + \|\eta\|^2} (1 + \frac{C}{\sqrt{I}})e^{-\frac{1}{I^2} + \frac{C_4}{I} }.
$$
To estimate $\|d \mathcal{I}\|$ (see (\ref{eq:def-norm-dI})), we 
take the supremum of $ \|d\mathcal{I}(\xi, \eta)\|$ over all $(\xi, \eta)\in T_xK^+ \oplus T_xK^+$ satisfying the condition $\|\xi\|^2 + \|\eta\|^2 = 1$. Noting that $\|\xi\| + \|\eta\| \le \sqrt{2}$, we obtain
\begin{equation}\label{eq:bound-est1}
\| d\mathcal{I} \| 
\le 
\sqrt{2}  e^{-\frac{1}{I^2}}
\left(
  \frac{2C_2}{I^3} (\sqrt{I}+1)
  + C_1(1 + \frac{C}{\sqrt{I}})
  \right)
  +  
  2(1 + \frac{C}{\sqrt{I}})
  e^{-\frac{1}{I^2} + \frac{C_4}{I}}
.
\end{equation}
As $(x, v) \to (x_0, v_0) \in \mathcal{S}^+$, we have $0 < I(x,v) \to 0$. The limit of the RHS of (\ref{eq:bound-est1}) is $0$ as $I\to 0$.
Consequently, $\|d\mathcal{I}\| \to 0$ as $(x, v) \to (x_0, v_0)$, which concludes the proof of Lemma \ref{lem:tangent-map-vanish}.

\end{proof}

Combining Lemmas~\ref{lem:diff-at-singular} and~\ref{lem:tangent-map-vanish}, we conclude that
\(\mathcal{I}\) is of class \(C^{1}\) on \(\mathcal{S}^+\).
Moreover, \(\mathcal{I}\) is \(C^{1}\)-smooth on \(\mathcal{T}^+\), since it is the product of the smooth function \(e^{-1/I^{2}}\) and the \(C^{1}\) map \(\mathcal{J}\) on \(\mathcal{T}^+\) (see Lemma~\ref{lem:smooth-integrals}).
Hence $\mathcal{I}$ is $C^1$-smooth on $T^1K^+$.
This completes the proof of part~2) of Lemma~\ref{lem:continuous-integrals}, and hence establishes Theorem~\ref{lem:integral}.

\section{Liouville--Arnold integrability}
In this section we will prove Theorem~\ref{thm:liouville-integrability}.
Let \(z=(z^1,\dots,z^n)\) be local coordinates on \(\Sigma\), and let
\[
\psi(z)= (\psi^1(z),\dots, \psi^{N+1}(z)) \in \Sigma \subset \mathbb{S}^{N} \subset \mathbb{R}^{N+1}
\]
be a parametrization of \(\Sigma\).
Let
\[
\Phi(t,z)= t\psi(z) \in K^+ \subset \mathbb{R}^{N+1},
\qquad t>0,
\]
be a parametrization of \(K^+\).
The metric on $K^+$ is given by
\[
dt^2 + t^2\sum_{i,j=1}^n g_{ij}(z)\, dz^i dz^j .
\]
Let \((y^0, y)\), with \(y=(y^1,\dots,y^n)\), denote the coordinates of a tangent
vector with respect to the basis
\(\{\partial_t, \partial_{z^1},\dots, \partial_{z^n}\}\),
and let \((p_0, p)\), with \(p=(p_1,\dots,p_n)\), denote the coordinates of a
cotangent vector with respect to the dual basis
\(\{dt, dz^1,\dots, dz^n\}\).

The Lagrangian of the geodesic flow on $TK^+$ is given by
\[
L(t, z; y^0, y) = \frac{1}{2}\left( (y^0)^{2} + t^2 \sum_{i,j=1}^n g_{ij}(z)\, y^i y^j \right).
\]
The \textit{Legendre transform} $\mathcal{L}$ associated with \(L\) is a map from
$TK^+$ to $T^*K^+$ that maps $T_{(t,z)}K^+$ to $T_{(t,z)}^*K^+$ via
\begin{equation}\label{eq:Legendre}
  p_0 = \frac{\partial L}{\partial y^0} = y^0,
  \qquad
  p_i = \frac{\partial L}{\partial y^i}
  = t^2 \sum_{j=1}^n g_{ij}(z)\, y^j .
\end{equation}
The Hamiltonian of the geodesic flow on \(T^*K^+\) is then given by
\begin{equation}\label{eq:Hamiltonian}
  H(t,z; p_0,p)
:= L\circ \mathcal{L}^{-1}
= \frac12 p_0^2 + \frac{1}{2t^2} \sum_{i,j=1}^n g^{ij}(z)\, p_i p_j,
\end{equation}
where $(g^{ij})$ denotes the inverse matrix of $(g_{ij})$.
Since the metric $g_{ij}$ is positive definite,
the Legendre transform $\mathcal{L}$ is a diffeomorphism from $TK^+$ to $T^*K^+$.
Let
\begin{equation}\label{eq:hat-tauplus}
  \widehat{\mathcal T}^{+}
  := \{(t,z; y^0, y)\in TK^+ \mid y\neq 0\}.
\end{equation}
Then $\widehat{\mathcal T}^{+} \subset TK^+$ is the union of all non-radial geodesic trajectories in $TK^+$.
One can check that
\[
\mathcal{M} = \mathcal{L}(\widehat{\mathcal T}^{+}) \subset T^*K^+.
\]

Recall that in Section~3 we define the first integrals
(see \eqref{eq:many-integrals1}):
\[
J_k \colon \mathcal T \to \mathbb R,
\qquad k = 1, \dots, N+1.
\]
Geometrically, for each $(x,v)\in \mathcal T$ with $\|v\|=1$, the geodesic
with initial data $(x,v)$ is tangent to the sphere $\mathbb S(\sqrt I)\subset\mathbb R^{N+1}$
at a unique point, and $J_k(x,v)$ is defined as the $k$-th coordinate of this point in $\mathbb R^{N+1}$.
This geometric construction extends without modification to arbitrary
$(x,v)\in \widehat{\mathcal T}^+$.
Indeed, for any nonzero initial velocity $v$, the geodesic with initial data $(x,v)$
is tangent to the sphere $\mathbb{S}(\sqrt I)$ at a unique point as well.
Since the geodesics with initial data $(x,v)$ and $(x,v/\|v\|)$ coincide as
unparametrized curves, their tangency points with $\mathbb{S}(\sqrt I)$ are the same.
We therefore define the extended first integrals
\[
\tilde J_k \colon \widehat{\mathcal T}^+ \to \mathbb R
\]
by the relation
\[
\tilde J_k(x,v) := J_k\!\left(x,\frac{v}{\|v\|}\right),
\qquad k = 1, \dots, N+1.
\]
In local coordinates $(t,z; y^0, y)$ on $TK^+$, this definition reads
\[
\tilde{J}_k(t,z; y^0, y)
:= J_k\!\left(
t,z;
\frac{y^0}{(2L)^{1/2}},
\frac{y}{(2L)^{1/2}}
\right),
\qquad k = 1, \dots, N+1.
\]

We then define
\[
\tilde F_k := \frac{\tilde J_k}{\sqrt I}
\colon \widehat{\mathcal T}^+ \to \mathbb R,
\qquad k = 1, \dots, N+1.
\]
Since both $\tilde J_k$ and $I$ are first integrals of the geodesic flow on
$\widehat{\mathcal T}^+$, it follows that the functions $\tilde F_k$ are also
first integrals.

\begin{remark}\label{rem:homogeneous}
  The first integrals $\tilde F_k$ are homogeneous of degree zero in $v$, that is,
  \[
  \tilde F_k(x,\lambda v)=\tilde F_k(x,v),
  \qquad \lambda>0.
  \]
  Indeed, the function $I(x,v)$ satisfies
  $
  I(x,\lambda v)=I(x,v),
  $
  and, by construction, the extended integrals $\tilde J_k$ satisfy the same
  property,
  $
  \tilde J_k(x,\lambda v)=\tilde J_k(x,v).
  $
  \end{remark}

Using the Legendre transform~(\ref{eq:Legendre}), we construct the functions
\[
F_k := \tilde{F}_k\circ \mathcal{L}^{-1} : \mathcal{M} \to \mathbb{R},
\qquad
I^* := I\circ \mathcal{L}^{-1} : \mathcal{M} \to \mathbb{R}.
\]
Then $F_k$ and $I^*$ are first integrals of the geodesic flow restricted to
$\mathcal{M}$.

Let us find $I^*$ in local coordinates.
In coordinates $(t,z; y^0, y)$, the geodesic $\gamma(s)$ and its tangent vector
$\gamma'(s)$ are represented as
\[
\begin{aligned}
  \gamma(s) &= t\, \psi(z), \\
  \gamma'(s)
  &= \frac{dt}{ds}\, \psi(z)
  + \sum_{i=1}^n t \frac{dz^i}{ds}\,
    \frac{\partial \psi(z)}{\partial z^i} = y^0 \psi(z)
  + \sum_{i=1}^n y^i \left(t\frac{\partial \psi(z)}{\partial z^i}\right).
\end{aligned}
\]
Thus by~(\ref{eq:distance-integral}) the first integral $I$ has the form
\[
\begin{aligned}
  I
  &= \|\gamma(s)\|^2
  - \frac{\langle \gamma(s), \gamma'(s)\rangle^2}
         {\|\gamma'(s)\|^2} \\
  &= t^2 \|\psi(z)\|
  - \frac{\left\langle
      t\psi(z),
      y^0 \psi(z)
      + t \sum_{i=1}^n
        \frac{\partial \psi(z)}{\partial z^i} y^i
    \right\rangle^2}
    {(y^0)^2 + t^2 \sum_{i,j=1}^n g_{ij}(z) y^i y^j} \\
  &= t^2
  - \frac{(t y^0)^2}
    {(y^0)^2 + t^2 \sum_{i,j=1}^n g_{ij}(z) y^i y^j}.
\end{aligned}
\]
Here we use $\|\psi(z)\|=1$,
$\left\langle \psi(z),
\frac{\partial \psi(z)}{\partial z^i} \right\rangle=0$,
and
\[
\left\langle
\frac{\partial \psi(z)}{\partial z^i},
\frac{\partial \psi(z)}{\partial z^j}
\right\rangle = g_{ij}.
\]
Via the Legendre transform~(\ref{eq:Legendre}) we obtain
\begin{equation}\label{eq:I-formula}
  I^* = I\circ \mathcal{L}^{-1}
  = t^2 - \frac{t^2 p_0^2}{2H}.
\end{equation}
One can easily check that $\{I^*,H\}=0$, where for differentiable
functions
$G_1,G_2$ on $T^*K^+$
\[
\{G_1,G_2\}
 =
 \frac{\partial G_2}{\partial t}\frac{\partial G_1}{\partial p_0}
 -
 \frac{\partial G_1}{\partial t}\frac{\partial G_2}{\partial p_0}
 +
 \sum_{i=1}^n\left(
 \frac{\partial G_2}{\partial z^i}\frac{\partial G_1}{\partial p_i}
 -
 \frac{\partial G_1}{\partial z^i}\frac{\partial G_2}{\partial p_i}
 \right).
\]

 We have the following.
\begin{lemma}\label{lem:poisson-brackets}
  The functions $F_k$ are in involution, i.e., $\{F_i, F_j\} = 0$ for all $1 \le i, j \le N+1$.
\end{lemma}

\begin{proof}
  The idea of the proof is the following. 
  The manifold $\mathcal{M}$ is covered by the trajectories that correspond to non-radial trajectories in $TK^+$.
  Since $F_i, i= 1, \dots, N+1$ are first integrals, 
  $$
F_{ij}:= \{F_i, F_j\},
\qquad  i, j =1,\dots, N+1,
  $$
  are also first integrals of the geodesic flow and hence is constant along the trajectories. 
  In $TK^+$, 
each non-radial trajectory has a special point, where the trajectory is tangent to the sphere $\mathbb{S}^N(\sqrt{I})$.
On the dual manifold $\mathcal{M}\subset T^*K^+$ every trajectory has also unique special point $A\in \mathcal{M}$, characterized by $p_0=0$, and this point corresponds to the touching point of the non-radial trajectory in $TK^+$ of $\mathbb{S}^N(\sqrt{I})$. To check that 
$F_{ij}=0$ we will make the calculations in a neighborhood of the point $A\in \mathcal{M}$.

Let us find the form of $F_k$ in a neighborhood of $A\in \mathcal{M}$.
For this we consider the geodesic $\gamma_{x,v}(s)$ with initial data $A=(x,v)$, where
$\langle x, v \rangle =0$.
In local coordinates the condition $\langle x, v \rangle =0$ is $y_0=0$, and we write 
\[
x = t\,\psi(z), \qquad 
v = \sum_{i=1}^n y^i \left(t \frac{\partial \psi}{\partial z^i}\right).
\]
The point
\[
\bigl(\tilde J_1(t,z;0,y), \dots, \tilde J_{N+1}(t,z;0,y)\bigr)
\]
is the unique point of tangency of the geodesic $\gamma_{x,v}(s)$ with the sphere $\mathbb S^N(\sqrt I)$.
Since $\gamma_{x,v}(s)$ is tangent to $\mathbb S^N(\sqrt I)$
precisely at $x = t\psi(z)$, it follows that
$$
  \tilde J_i \big|_{y^0 = 0} = t\,\psi^i(z),
  \qquad i = 1, \dots, N+1.
$$
Hence, via the Legendre transform~\eqref{eq:Legendre}, we obtain
\begin{equation}\label{eq:rj-cot}
  \left(\tilde{J}_i \circ \mathcal{L}^{-1}\right)\big|_{p_0 = 0}
  = t\,\psi^i(z).
\end{equation}
By~\eqref{eq:I-formula} and~\eqref{eq:rj-cot}, it follows that
\[
F_i\big|_{p_0 = 0}
= \frac{\left(\tilde{J}_i \circ \mathcal{L}^{-1}\right)\big|_{p_0 = 0}}
       {\sqrt{I^*}\big|_{p_0 = 0}}
= \psi^i(z),
\qquad
i = 1, \dots, N+1.
\]
In particular, $F_i\big|_{p_0 = 0}$ is independent of both $t$ and $p$.
Hence
$$
\frac{\partial F_i}{\partial t} \Big|_{p_0=0}
=0, 
\qquad
\frac{\partial F_i}{\partial p_k} \Big|_{p_0=0}
=
0, \; 
k=1,\dots,n.
$$
Therefore 
$$
    F_{ij}(A)=\{F_i, F_j\}|_{p_0=0} = 0.
$$
Since $F_k$ are first integrals, their Poisson brackets $\{F_i, F_j
\}$ are constant along geodesic trajectories.
Therefore, we obtain
$$
    \{F_i, F_j\} =  0.
$$
Lemma \ref{lem:poisson-brackets} is proved.
\end{proof}

By Lemma~\ref{lem:poisson-brackets}, we obtain $N+1$ first integrals in involution.
For Liouville--Arnold integrability, we require $n+1$ first integrals, including the Hamiltonian $H$, that are in involution and functionally independent almost everywhere.
To construct such  integrals,
in Lemma~\ref{lem:rank-F} we prove that the rank of the map $(F_1,\dots,F_{N+1})$ is $n$, and in Lemma~\ref{lem:generic-G} we construct the required first integrals.

\begin{lemma}\label{lem:rank-F}
  The rank of the map $F = (F_1, \dots, F_{N+1}): \mathcal{M} \to \Sigma \subset \mathbb{R}^{N+1}$ is $n=\dim \Sigma$ at every point.
  \end{lemma}
  
\begin{proof}
  Recall that $\widehat{\mathcal T}^+$ denotes the open subset of $TK^+$
  corresponding to non-radial geodesics (see~\eqref{eq:hat-tauplus}),
  and $\mathcal T^+$ denotes the open subset of $T^1K^+$
  corresponding to unit-speed (i.e., $\|v\|=1$) non-radial geodesics (see~\eqref{eq:tauplus}).

    We begin by decomposing the map $F$ as a composition of smooth maps:
$$
      F = \mu \circ \Psi \circ \tau \circ \mathcal{L}^{-1} \colon {\mathcal{M}} \to \Sigma \subset \mathbb{R}^{N+1},
$$
    where
    \[
    \mathcal{L}^{-1} \colon \mathcal{M} \to \widehat{\mathcal T}^+ \subset 
    TK^+
    \]
    is the inverse Legendre transform restricted to $\mathcal{M}$, and
    \[
    \tau \colon\widehat{\mathcal T}^+ \to \mathcal{T}^+ \subset T^1K^+,
    \qquad
    \tau(x,v) = \bigl(x, \tfrac{v}{\|v\|}\bigr),
    \]
   $$
      \Psi \colon \mathcal{T}^+ \to K^+,
      \qquad
      \Psi(x,v) = \gamma_{x,v}\bigl(s_0(x,v)\bigr) 
      \in 
      \mathbb{S}^N(\sqrt{I})\cap K^+,
$$
    $$
    \mu: K^{+} \to \Sigma \subset \mathbb{R}^{N+1},
    \qquad \mu(x) = \frac{x}{\|x\|}.
    $$
    Since the image of $F$ is contained in
    $\Sigma \subset \mathbb{R}^{N+1}$ and $\dim \Sigma = n$, the rank of $F$ is at most $n$.
    To show that $\operatorname{rank} F = n$, it suffices to prove that
    \begin{itemize}
      \item[(i)] $\mathcal{L}^{-1} \colon \mathcal{M} \to \widehat{\mathcal T}^+$ is a submersion (i.e., $d(\mathcal{L}^{-1})$ is surjective at every point of $\mathcal{M}$);
      \item[(ii)] $\tau \colon \widehat{\mathcal T}^+ \to \mathcal{T}^+ $ is a submersion;
      \item[(iii)] $\Psi \colon \mathcal{T}^+ \to K^+$ is a submersion;
      \item[(iv)] $\mu: K^{+} \to \Sigma \subset \mathbb{R}^{N+1}$ is a submersion to its image $\Sigma$.
    \end{itemize}
  
    The statements (i), (ii), (iv) follow directly from the definitions of the maps. Let us prove (iii).

    Let $\mathcal{Z} \in T_{(x,v)}\mathcal{T}^+$ be a tangent vector at $(x,v)$ with $\|v\|=1$, represented via the canonical isomorphism by
    $$
    (\xi, \eta)\in T_xK^+\oplus T_xK^+,
    $$
with
    \begin{equation}\label{eq:domain-constraint}
      \langle v, \eta \rangle=0,
    \end{equation}
    (see \eqref{eq:decomposition1} and \eqref{eq:decomposition1-t1}). 
    Using \eqref{eq:diff-pos} and (\ref{eq:ds0}), the tangent map $d\Psi: T_{(x,v)}(\mathcal{T}^+)\to T_{\Psi(x,v)}K^+$ is given by:
    \begin{equation}\label{eq:dPsi_vector}
        d\Psi(\xi, \eta) 
        =
        \frac{d}{d\varepsilon}\Big|_{\varepsilon=0} 
        \gamma_{x(\varepsilon), v(\varepsilon)}(s_0(\varepsilon))
        = - \left( \langle \xi, v \rangle +\langle x, \eta \rangle \right) \gamma'_{x,v}(s_0) + J(s_0) 
        \quad
        \in T_{\Psi(x,v)}K^+,
    \end{equation}
    where $J(s)$ is the Jacobi field along $\gamma_{x,v}(s)$ satisfying $J(0)=\xi$ and $\frac{D}{ds}J(0)=\eta$ (see (\ref{eq:jacobi-xi-eta})).
  
    \medskip
    Before proving the  surjectivity of $d\Psi$, we establish some properties of Jacobi fields on $K^+$ that will be used later. Let $J(s)$ be any Jacobi field along $\gamma_{x,v}(s)$. Consider the functions:
    $$
      h(s):= \langle \frac{D}{ds}J(s), \gamma'_{x,v}(s)\rangle, \quad
      l(s):= \langle \frac{D}{ds}J(s), \gamma_{x,v}(s)\rangle, \quad
      g(s):= \langle J(s), \gamma'_{x,v}(s)\rangle.
    $$
    
    First, for $h(s)$, using $\frac{D}{ds}\gamma'_{x,v}(s)=0$ and (\ref{eq:j-equ}) we have 
    \begin{equation}\label{eq:hs-const}
      \frac{d}{ds}h(s) = 
      \langle \frac{D^2}{ds^2}J(s), \gamma'_{x,v}(s)\rangle + 
      \langle\frac{D}{ds} J(s), \frac{D}{ds}\gamma'_{x,v}(s)\rangle  =  \langle -  R\big(J(s),\gamma'_{x,v}(s)\big)\gamma'_{x,v}(s), \gamma'_{x,v}(s)\rangle= 0.
    \end{equation}
    Therefore, $h(s)$ is constant along $\gamma_{x,v}(s)$.
  
    Second, for $l(s)$, using $\frac{D}{ds}\gamma'_{x,v}(s)=0$ and (\ref{eq:j-equ}) we have 
    $$
    \frac{d}{ds}l(s) = 
    \langle \frac{D^2}{ds^2}J(s), \gamma_{x,v}(s)\rangle + 
    \langle \frac{D}{ds} J(s), \gamma'_{x,v}(s)\rangle 
    =\langle -  R\big(J(s),\gamma'_{x,v}(s)\big)\gamma'_{x,v}(s), \gamma_{x,v}(s)\rangle +  h(s).
    $$
    By the interchange symmetry of Riemann curvature tensor and part~1) of Lemma~\ref{lem:cone-curvature-formula}, we have 
    $$
    \langle -  R\big(J(s),\gamma'_{x,v}(s)\big)\gamma'_{x,v}(s), \gamma_{x,v}(s)\rangle = \langle -  R\big(\gamma'_{x,v}(s), \gamma_{x,v}(s)\big)  J(s),\gamma'_{x,v}(s)\rangle =0.
    $$
    Hence 
    \begin{equation}\label{eq:ls-deriv}
      \frac{d}{ds}l(s) = h(s).
    \end{equation}
  
    Third, for $g(s)$, using $\frac{D}{ds}\gamma'_{x,v}(s)=0$ and (\ref{eq:j-equ}) we have 
    \begin{equation}\label{eq:gs-deriv}
      \frac{d}{ds}g(s) =  \langle \frac{D}{ds}J(s), \gamma'_{x,v}(s)\rangle + 
      \langle J(s), \frac{D}{ds}\gamma'_{x,v}(s)\rangle  =  \langle \frac{D}{ds}J(s), \gamma'_{x,v}(s)\rangle= h(s).
    \end{equation}
    In particular, if $h(s) \equiv 0$, then $l(s)$ and $g(s)$ are constant.
  
    \medskip
    Now let us prove that $d\Psi: T_{(x,v)}\mathcal{T}^+ \to T_{\Psi(x,v)}K^+$ is surjective.
       The tangent space $T_{\Psi(x,v)}K^+$ admits the orthogonal decomposition
       \[
       T_{\Psi(x,v)}K^+
       = \operatorname{span}\{\gamma_{x,v}(s_0)\}
       \oplus \operatorname{span}\{\gamma_{x,v}'(s_0)\}
       \oplus W,
       \]
       where $W$ denotes the orthogonal complement of
       $
       \operatorname{span}\{\gamma_{x,v}(s_0)\}
       \oplus \operatorname{span}\{\gamma_{x,v}'(s_0)\}
       $
       in $T_{\Psi(x,v)}K^+$. We prove the surjectivity in two steps.
   
    \textit{Step 1: Prove $\operatorname{span}\{\gamma_{x,v}'(s_0)\}
    \oplus W \subset \operatorname{Im} d \Psi$.}

    Let $Y_1\in W, \alpha_1 \in \mathbb{R}$.
    Let $J(s)$ be the Jacobi field 
    along $\gamma_{x,v}(s)$ satisfying
    \begin{equation}\label{eq:jacobi-1}
      J(s_0)=Y_1, 
    \qquad 
    \frac{D}{ds}J(s_0)= \alpha_1 \gamma_{x,v}(s_0). 
    \end{equation}
    Let us define $(\xi_1,\eta_1)\in T_xK^+ \oplus  T_xK^+$ as 
    $$
    (\xi_1,\eta_1):= \left(J(0), \frac{D}{ds}J(0)\right).
    $$
    We will verify $
    (\xi_1,\eta_1)\in T_{(x,v)} \mathcal{T}^+$ by checking $\langle \eta_1, v \rangle=0$, and will show that 
    \begin{equation}\label{eq:image-1}
    d\Psi(\xi_1, \eta_1) = -I \alpha_1 \gamma_{x,v}'(s_0) + Y_1.
  \end{equation}

    Using the properties established in (\ref{eq:hs-const}), (\ref{eq:ls-deriv}) and (\ref{eq:gs-deriv}):
    \begin{itemize}
      \item Since $h(s)$ is constant, using $h(0)=h(s_0)$ we have 
      $$
      \langle \eta_1,v \rangle = \langle \frac{D}{ds}J(0), \gamma'_{x,v}(0)\rangle=\langle \frac{D}{ds}J(s_0), \gamma'_{x,v}(s_0)\rangle=
      \langle \alpha_1 \gamma_{x,v}(s_0), \gamma'_{x,v}(s_0) \rangle=0.
      $$ This verifies the condition (\ref{eq:domain-constraint}), hence verifies that $(\xi_1,\eta_1)\in T_{(x,v)} \mathcal{T}^+$.
      \item Since $h(s_0)=0$, we have $h(s)\equiv 0$. Thus $l(s)$ is constant, which gives
      \begin{equation}\label{eq:ls01}
        \langle \eta_1,  x  \rangle = l(0)=l(s_0)= 
        \langle \alpha_1 \gamma_{x,v}(s),\gamma_{x,v}(s) \rangle = \alpha_1 I.
      \end{equation}
      \item Since $h(s)\equiv 0$, $g(s)$ is constant, which gives
      \begin{equation}\label{eq:gs01}
        \langle \xi_1, v\rangle= g(0)= g(s_0)= \langle Y_1, \gamma'_{x,v}(s_0)\rangle=0 .
      \end{equation}
    \end{itemize}
    Using (\ref{eq:gs01}) and (\ref{eq:ls01}), we have 
    \begin{equation}\label{eq:ds0-g1}
      \langle \xi_1, v \rangle + \langle x, \eta_1 \rangle  = \alpha_1 I.
    \end{equation}
    Therefore, by (\ref{eq:dPsi_vector}), (\ref{eq:jacobi-1}) and (\ref{eq:ds0-g1}) we obtain (\ref{eq:image-1}):
    $$
    d\Psi(\xi_1, \eta_1) = - \alpha_1 I\gamma_{x,v}'(s_0) + Y_1.
    $$
    Hence $\operatorname{span}\{\gamma_{x,v}'(s_0)\}\oplus W
    \subset \operatorname{Im} d\Psi$.

    \textit{Step 2: Prove $\operatorname{span}\{\gamma_{x,v}(s_0)\} \subset \operatorname{Im} d \Psi$.}
    
    Consider the vector field
    \begin{equation}\label{eq:jacobi-2}
      X(s) := \gamma_{x,v}(s) - s\,\gamma'_{x,v}(s)
    \end{equation}
    along the geodesic $\gamma_{x,v}(s)$.
    Then $X(s)$ is a Jacobi field. Indeed, a direct computation shows that
    \[
    \frac{D}{ds} X(s)
    = \gamma'_{x,v}(s) - \gamma'_{x,v}(s)
    - s \frac{D}{ds} \gamma'_{x,v}(s)
    = 0,
    \]
    and, using part~1) of Lemma~\ref{lem:cone-curvature-formula},
    \[
    R\big(X(s),\gamma'_{x,v}(s)\big)\gamma'_{x,v}(s)
    =
    R\big(\gamma_{x,v}(s),\gamma'_{x,v}(s)\big)\gamma'_{x,v}(s)
    -
    R\big(s \gamma'_{x,v}(s),\gamma'_{x,v}(s)\big)\gamma'_{x,v}(s)
    = 0.
    \]
    Hence,
    \[
    \frac{D^2}{ds^2} X(s)
    +
    R\big(X(s),\gamma'_{x,v}(s)\big)\gamma'_{x,v}(s)
    = 0.
    \]
    
    Let $(\xi_2,\eta_2) := (X(0), \tfrac{D}{ds}X(0))$. Then
    \[
    \xi_2 = \gamma_{x,v}(0)=x,
    \quad
   \eta_2 =  0.
    \]
    In particular, we have $\langle v,\eta_2 \rangle = \langle v,0 \rangle = 0$,
    which shows that $(\xi_2,\eta_2) \in T_{(x,v)}\mathcal{T}^+$.

    Using (\ref{eq:dPsi_vector}), (\ref{eq:jacobi-2}) and $s_0=- \langle x, v \rangle$ we obtain 
    $$
    \begin{aligned}
      d\Psi(\xi_2,\eta_2)& =  - \left( \langle \xi_2, v \rangle +\langle x, \eta_2 \rangle \right) \gamma'_{x,v}(s_0) + X(s_0)  \\
    & = - \left( \langle x, v \rangle +\langle x, 0\rangle \right) \gamma'_{x,v}(s_0) + \gamma_{x,v}(s_0)-s_0 \gamma_{x,v}'(s_0)\\
    &= - \langle x, v \rangle  \gamma'_{x,v}(s_0)+ \gamma_{x,v}(s_0) + \langle x, v \rangle  \gamma'_{x,v}(s_0) \\
    & = \gamma_{x,v}(s_0).
    \end{aligned}
    $$
    Hence $\operatorname{span}\{\gamma_{x,v}(s_0)\}\in \operatorname{Im} d\Psi$.
  
    Therefore, we have proved the tangent map $d\Psi$ is surjective, and hence $\Psi$ is a submersion.

    \medskip

    With (iii) established, and using that (i), (ii), and (iv) are submersions, we conclude that the composition 
$$
F = \mu \circ \Psi \circ \tau \circ \mathcal{L}^{-1}: \mathcal{M} \to \Sigma
$$
is a submersion, as the composition of submersions is a submersion. Since $\dim \Sigma = n$, we have 
$$
\operatorname{rank} F = n.
$$
  \end{proof}
 
  In the following lemma, we aim to construct a map $\mathcal{F}=(\mathcal{F}_1,\dots,\mathcal{F}_n): \mathcal{M} \to \mathbb{R}^n$ consisting of $n$ functions (first integrals) that are functionally independent almost everywhere and in involution. To establish the existence of such a map, we will employ the Thom Transversality Theorem (see, e.g, \cite{GG}).

\begin{lemma}\label{lem:generic-G}
  There exists a $C^1$ map $G \colon \Sigma \to \mathbb{R}^n$ such that the composition
  \[
  \mathcal{F} := G \circ F \colon \mathcal{M} \to \mathbb{R}^n
  \]
  has rank $n$ almost everywhere on $\mathcal{M}$.
  \end{lemma}
  \begin{proof}
    Since every $C^r$, $1\le r < \infty$, manifold is $C^r$-diffeomorphic
    to a $C^\infty$ manifold (see, e.g., \cite[Chapter~2, Theorem~2.10]{H}),
    we fix a $C^3$-diffeomorphism
    \[
    \Theta \colon \Sigma \longrightarrow \tilde{\Sigma}
    \]
    onto a $C^\infty$ manifold $\tilde{\Sigma}$.
    We will apply Thom transversality theorem to obtain 
a $C^{\infty}$ map
\[
\tilde{G} \colon \tilde{\Sigma} \to \mathbb{R}^n
\]
whose differential has rank $n$ almost everywhere.
Then the composition
\[
G = \tilde{G} \circ \Theta \colon \Sigma \to \mathbb{R}^n
\]
is $C^3$ smooth,
and in particular $C^1$ smooth,
which is the desired map.

Consider the 1-jet bundle $J^1(\tilde{\Sigma}, \mathbb{R}^n)$. Let us briefly recall the definition (see, e.g., 
\cite[Chapter~II, Definition~2.1]{GG}). 
Two smooth mappings $f, g: \tilde{\Sigma} \to \mathbb{R}^n$ are said be \emph{equivalent} at a point $p \in \tilde{\Sigma}$, denoted by $f \sim_1 g$, if
\[
    f(p) = g(p) \quad \text{and} \quad (df)_p = (dg)_p,
\]
where $(df)_p: T_p\tilde{\Sigma} \to T_{f(p)}\mathbb{R}^n$ is the differential at $p$. 
Let $[f]_p$ denote the equivalence class of $f$ under the equivalence relation $\sim_1$. The 1-jet bundle is defined as the union of such equivalence classes:
\[
    J^1(\tilde{\Sigma}, \mathbb{R}^n) = \bigcup_{(p,q) \in \tilde{\Sigma}\times \mathbb{R}^n} \left\{ [f]_p \mid f: \tilde{\Sigma} \to \mathbb{R}^n \text{ is smooth}, f(p)=q \right\}.
\]
$J^1(\tilde{\Sigma}, \mathbb{R}^n)$ forms a vector bundle over the product $\tilde{\Sigma} \times \mathbb{R}^n$. An element $\sigma = [f]_p \in J^1(\tilde{\Sigma}, \mathbb{R}^n)$ is uniquely characterized by the triple $(p, q, L)$, where $p \in \tilde{\Sigma}$, $q = f(p) \in \mathbb{R}^n$, and $L = (df)_p$.

Let $S \subset J^1(\tilde{\Sigma}, \mathbb{R}^n)$ be the \emph{singular locus}, defined as the set of 1-jets $\sigma = (p, q, L)$ where the linear map $L$ is singular (i.e., $\operatorname{rank} L < n$). 
For each $r \in \{1, \dots, n\}$, define
\[
    S_r = \{ \sigma = (p, q, L) \in S \mid \operatorname{corank}(L) = r \}.
\]
Following \cite[Chapter~II, Theorem~5.4]{GG}, each  $S_r$ is a submanifold of codimension $r^2$ in $J^1(\tilde{\Sigma}, \mathbb{R}^n)$. Thus, $S$ can be written as a union of submanifolds
\[
    S = \bigcup_{r=1}^n S_r.
\]

For any smooth map $f: \tilde{\Sigma} \to \mathbb{R}^n$, we define its \emph{1-jet extension} $j^1f: \tilde{\Sigma} \to J^1(\tilde{\Sigma}, \mathbb{R}^n)$ by
\[
    j^1f(p) = \left( p, f(p), (df)_p \right).
\]
The map $j^1f$ is said to be \emph{transverse} to the submanifold $S_r$ if, for every point $p$ such that the jet $\sigma = j^1f(p)$ lies in $S_r$, the tangent space $T_{\sigma}S_r$ and the tangent space to the image $T_{\sigma}(j^1f(\tilde{\Sigma}))$ span the entire tangent space of the jet bundle at $\sigma$.

Consider the space $C^\infty(\tilde{\Sigma}, \mathbb{R}^n)$ of all smooth mappings equipped with the $C^\infty$-topology.
The Thom Transversality Theorem (see \cite[Chapter~II, Theorem~4.9]{GG}) implies that for each $r$, the set of maps $f: \tilde{\Sigma} \to \mathbb{R}^n$ whose 1-jet extension $j^1f$ is transverse to $S_r$ is a   \emph{residual} subset (see \cite[Chapter~II, Definition~3.2]{GG}) of $C^\infty(\tilde{\Sigma}, \mathbb{R}^n)$; thus this set is dense in $C^\infty(\tilde{\Sigma}, \mathbb{R}^n)$  (see \cite[Chapter~II, Proposition~3.3]{GG}). 
 Since the intersection of finitely many dense subsets is dense, there exists a map $\tilde{G} \in C^\infty(\tilde{\Sigma}, \mathbb{R}^n)$ such that $j^1\tilde{G}$ is transverse to every $S_r$ simultaneously.

For such a generic map $\tilde{G}$, we define the singular locus of $\tilde{G}$ as the preimage of $S$ under $j^1\tilde{G}$:
\[
    \mathcal{K}_{\tilde{G}} := (j^1\tilde{G})^{-1}(S) = \bigcup_{r=1}^n (j^1\tilde{G})^{-1}(S_r).
\]
Following \cite[Chapter~II, Theorem~4.4]{GG}, each preimage $(j^1\tilde{G})^{-1}(S_r)$ is a submanifold of $\tilde{\Sigma}$ with codimension $r^2$.  
Thus, $\mathcal{K}_{\tilde{G}}$ is a finite union of submanifolds of positive codimension. Therefore, $\mathcal{K}_{\tilde{G}}$ has measure zero in $\tilde{\Sigma}$.

Finally, let $G = \tilde{G} \circ \Theta$. 
Since $\Theta: \Sigma \to \tilde{\Sigma}$ is a diffeomorphism, the singular locus of $G$, 
defined as
\[
    \mathcal{K}_G := \{ p \in \Sigma \mid \operatorname{rank}(dG)_p < n \}
\]
is precisely the preimage $\Theta^{-1}(\mathcal{K}_{\tilde{G}})$.
Since diffeomorphisms preserve sets of measure zero, $\mathcal{K}_G$ has measure zero in $\Sigma$.

Now, consider the composition $\mathcal{F} = G \circ F$. By the chain rule, for $a \in \mathcal{M}$,
\[
    (d\mathcal{F})_a = (dG)_{F(a)} \circ (dF)_a.
\]
Since $(dF)_a$ is surjective (Lemma~\ref{lem:rank-F}), we have
\[
    \operatorname{rank} (d\mathcal{F})_a = \operatorname{rank} (dG)_{F(a)}.
\]
Therefore, 
 the set 
 $$ \mathcal{K}_{\mathcal{F}}:=\{a \in \mathcal{M} \mid \operatorname{rank} (d\mathcal{F})_a < n\}
 $$ is exactly $F^{-1}(\mathcal{K}_G)$.
Since $F$ is a submersion and $\mathcal{K}_G$ has measure zero in $\Sigma$, it follows that $\mathcal{K}_{\mathcal{F}}$
has measure zero in  $\mathcal{M}$,
     which means $\mathcal{F}$ has rank $n$ almost everywhere.

    Lemma~{\ref{lem:generic-G}} is proved.
  \end{proof}

  With the map $\mathcal{F} = (\mathcal{F}_1, \dots, \mathcal{F}_n)$ constructed in Lemma~\ref{lem:generic-G}, we now show that $\{H, \mathcal{F}_1, \dots, \mathcal{F}_n\}$ constitutes the required set of $n+1$ functionally independent first integrals, where H is the Hamiltonian 
  (\ref{eq:Hamiltonian}).
  
  \begin{lemma}\label{lem:independence-with-H}
    The functions $(H, \mathcal{F}_1, \dots, \mathcal{F}_n)$ are functionally independent, i.e., 
    $\operatorname{rank}(H, \mathcal{F}_1, \dots, \mathcal{F}_n) = n+1$, 
    almost everywhere on $\mathcal{M}$. 
  \end{lemma}
  
  \begin{proof}
    By Lemma~\ref{lem:generic-G}, the Jacobian matrix of the map $\mathcal{F} = (\mathcal{F}_1, \dots, \mathcal{F}_n)$ has rank $n$ almost everywhere on $\mathcal{M}$. Therefore, to prove that the rank of the Jacobian matrix of the augmented map $(H, \mathcal{F}_1, \dots, \mathcal{F}_n)$ is $n+1$, it suffices to show that the $1$-form
    $dH$ does not lie in the span of $\{d\mathcal{F}_1, \dots, d\mathcal{F}_n\}$.
    
    To this end, let us consider the vector field on $T^*M$:
    $$
    Y := \sum_{j=0}^{n} p_j \frac{\partial}{\partial p_j}.
    $$   
    The Hamiltonian $H$ is homogeneous in $(p_0,p)$ of degree 2, implying
    \begin{equation}\label{eq:homogeneous-2}
      dH(Y) = 2H.
    \end{equation}
    In contrast,  since the map $F$ satisfies (see Remark~\ref{rem:homogeneous}) 
$$
        F(z,t; \lambda p_0, \lambda p ) = 
        F(z,t;  p_0,  p) \qquad 
        \forall \lambda > 0,
$$
    and the map $\mathcal{F}$ is defined by the composition $\mathcal{F} = G \circ F$, 
    we know
    $$
        \mathcal{F}_i(z,t; \lambda p_0, \lambda p) = G_i(F(z,t; \lambda p_0, \lambda p)) = G_i(F(z,t;  p_0,  p)) = \mathcal{F}_i(z,t;  p_0,  p).
    $$
    This implies that each $\mathcal{F}_i$ is homogeneous in $(p_0,p)$ of degree 0. By Euler's homogeneous function theorem, we have
    \begin{equation}\label{eq:homogeneous-0}
        d\mathcal{F}_i(Y) = 0, 
        \qquad  i = 1, \dots, n.
    \end{equation}
    Now, if 
    $dH \in \operatorname{span}\{d\mathcal{F}_1, \dots, d\mathcal{F}_n\}$, we can represent $dH$ as 
    \begin{equation}\label{eq:linear-independent}
      dH = \sum_{i=1}^{n} c_i\; d\mathcal{F}_i.
    \end{equation}
    Evaluating (\ref{eq:linear-independent}) on the  vector field $Y$ and using (\ref{eq:homogeneous-0}), we obtain:
$$
       dH(Y)=  \sum_{i=1}^{n} c_i d\mathcal{F}_i(Y) =0,
$$    
contradicting to (\ref{eq:homogeneous-2}).     
    Hence, $dH$ does not lie in the span of $\{d\mathcal{F}_1, \dots, d\mathcal{F}_n\}$, and hence 
    $$
    \operatorname{rank}(H, \mathcal{F}_1, \dots, \mathcal{F}_n)  = 
    \operatorname{rank}(\mathcal{F}_1, \dots, \mathcal{F}_n) +1 
    = n+1,
    $$
    almost everywhere on $\mathcal{M}$.

    Lemma~\ref{lem:independence-with-H} is proved.
  \end{proof}

  With Lemma~\ref{lem:poisson-brackets}-\ref{lem:independence-with-H} established, the proof of Theorem 2 is complete.

\vspace{1cm}

\noindent
{Andrey E. Mironov}\\
Sobolev Institute of Mathematics, Novosibirsk, Russia\\ 
Novosibirsk State University, Novosibirsk, Russia\\ 
Email: \texttt{mironov@math.nsc.ru}

\medskip
\noindent
{Siyao Yin}\\
Sobolev Institute of Mathematics, Novosibirsk, Russia\\
Email: \texttt{siyao.yin@math.nsc.ru}
\vspace{1cm}


\begin{thebibliography}{1}
	
	

\bibitem{BKF}
A. V. Bolsinov, V. V. Kozlov, A. T. Fomenko, The Maupertuis principle and geodesic flows on the sphere arising from integrable cases in the dynamics of a rigid body, Russian Math. Surveys, 50:3 (1995), 473--501.

\bibitem{M}
V. S. Matveev, Real Analyticity of 2-Dimensional Superintegrable Metrics and Solution of Two Bolsinov--Kozlov--Fomenko Conjectures, Regul. Chaotic Dyn., 30:4 (2025),  677--687

\bibitem{MS}
V. S. Matveev, V. V. Shevchishin, Differential invariants for cubic integrals of geodesic flows on surfaces, J. Geom. Phys., 60:6--8 (2010), 833--856.


\bibitem{V1}
G. Valent, On a class of integrable systems with a cubic first integral, Comm. Math. Phys., 299:3 (2010), 631--649.



\bibitem{VDS}
G. Valent, Ch. Duval, V. Shevchishin, Explicit metrics for a class of two-dimensional cubically superintegrable systems, J. Geom. Phys., 87 (2015), 461--481.


\bibitem{K1}
V. V. Kozlov, Integrability and non-integrability in Hamiltonian mechanics , Russian Math. Surveys, 38:1 (1983), 1--76

\bibitem{K}
K. Kiyohara, Two-dimensional geodesic flows having first integrals of higher degree, Math. Ann., 320:3 (2001), 487--505.

\bibitem{BF}
A. V. Bolsinov, A. T. Fomenko, Integrable geodesic flows on two-dimensional surfaces, Monogr. Contemp. Math., Consultants Bureau, New York (2000).

\bibitem{K}
B. Kruglikov, Invariant characterization of Liouville metrics and polynomial integrals, J. Geom. Phys., 58:8 (2008), 979--995.

\bibitem{Bel}
G. V. Belozerov, Geodesic flow on an intersection of several confocal quadrics in $\mathbb{R}^n$, Sbornik: Mathematics, 214:7 (2023), 897--918.
  
\bibitem{BF}
G. V. Belozerov, A. T. Fomenko, Generalized Jacobi--Chasles theorem in non-Euclidean spaces, Sbornik: Mathematics, 215:9 (2024), 1159--1181.

  
\bibitem{K}
V. V. Kozlov, Topological obstructions to the integrability of natural mechanical systems, Sov. Math. Dokl., 20 (1979), 1413--1415.

\bibitem{T}  
I. A. Taimanov, Topological obstructions to integrability of geodesic flows on non-simply-connected manifolds, Math. USSR-Izv., 30:2 (1988), 403--409.


\bibitem{But}  
  L. Butler, A new class of homogeneous manifolds with Liouville-integrable geodesic flows, C. R. Math. Rep. Acad. Sci. Canada, 21:43 (1999), 7509--7522.


  \bibitem{BT}  
  A. V. Bolsinov, I. A. Taimanov, Integrable geodesic flows with positive topological entropy, Invent. Math., 140 (2000), 639--650.

  \bibitem{BM}
  B. Bialy, A. E. Mironov,
Rich quasi-linear system for integrable geodesic flows on the 2-torus,
Discrete Contin. Dyn. Syst. A, 29 (2011), 81--90.

  \bibitem{MY1} 
  A. E. Mironov, S. Yin, Integrable Birkhoff billiards inside cones, arXiv:2501.12843 (2025).

  \bibitem{MY2} 
  A. E. Mironov, S. Yin, Billiard Trajectories inside Cones, Regul. Chaot. Dyn. 30 (2025), 688--710 .

  \bibitem{ON}
  B. O'Neill, Semi-Riemannian Geometry with Applications to Relativity. Academic Press. (1983) 

  \bibitem{Pat}
G. P. Paternain, Geodesic Flows. Birkh\"auser Boston. (1999)

\bibitem{GG}
M. Golubitsky, V. Guillemin, Stable Mappings and Their Singularities.
Springer-Verlag, New York. (1973)

\bibitem{H}
M. W. Hirsch, Differential Topology.
Springer-Verlag, New York. (1976)


\end{thebibliography}
\end{document}